\documentclass[a4paper,11pt]{article}
\usepackage[T1]{fontenc}
\usepackage{lmodern}
\usepackage[utf8]{inputenc}
\usepackage[USenglish]{babel}
\usepackage[babel=true]{microtype}
\usepackage{amsmath}
\usepackage{amsthm}
\usepackage{mathtools}
\usepackage{empheq}
\usepackage{amssymb}
\usepackage{commath}
\usepackage{mathrsfs}
\usepackage[pdftex]{graphicx}
\usepackage{fancyhdr}
\usepackage{thm-restate}
\usepackage[disable,colorinlistoftodos,color=yellow,textsize=small]{todonotes}
\usepackage{float}
\usepackage[numbers, sort]{natbib}
\usepackage{xpatch}
\usepackage{xparse}
\usepackage{hyperref}
\usepackage[labelfont=bf]{caption}
\usepackage{subcaption}
\usepackage[capitalise]{cleveref}
\usepackage[all]{hypcap} 
\usepackage[inline]{enumitem}
\usepackage{bm} 
\usepackage[normalem]{ulem}
\usepackage{typearea}

\hypersetup{%
pdftitle=,
pdfauthor=Max Reppen,
pdfcreator=,
pdfproducer=,
colorlinks=true,
linkcolor=black,
citecolor=black,
filecolor=black,
urlcolor=black,
bookmarksopen,
bookmarksopenlevel=1,
}

\makeatletter
\newtheoremstyle{indented}
  {3pt}
  {3pt}
  {\addtolength{\@totalleftmargin}{2em}
   \addtolength{\linewidth}{-2em}
   \parshape 1 2em \linewidth}
  {}
  {\bfseries}
  {.}
  {.5em}
  {}
\makeatother
\theoremstyle{indented}
\newtheorem{theorem}{Theorem}[section]
\newtheorem{corollary}{Corollary}[section]
\newtheorem{lemma}[theorem]{Lemma}
\newtheorem{proposition}[theorem]{Proposition}
\newtheorem{definition}[theorem]{Definition}
\newtheorem{remark}[theorem]{Remark}
\newtheorem{example}[theorem]{Example}

\newcounter{proofcount}
\newcounter{proofstep}[proofcount]
\AtBeginEnvironment{proof}{\refstepcounter{proofcount}\NewDocumentCommand{\step}{o}{\refstepcounter{proofstep}\emph{Step \arabic{proofstep}\IfNoValueTF{#1}{. }{: #1}}}}
\newcommand{\stepref}[1]{\emph{Step \ref{#1}}}

\DeclareMathOperator{\Tr}{Tr}

\newcommand{\esssup}{\operatornamewithlimits{ess\,sup}}
\newcommand{\essinf}{\operatornamewithlimits{ess\,inf}}
\newcommand{\eqdef}{:=}

\newcommand{\R}{\mathbb{R}}

\newcommand{\N}{\mathbb{N}}
\newcommand{\E}{\mathbb{E}}

\newcommand{\interior}[1]{%
  {\kern0pt#1}^{\mathrm{o}}%
}
\newcommand{\closure}[2][3]{{}\mkern#1mu\overline{\mkern-#1mu#2}}

\makeatletter
\def\1{\@ifnextchar[{\@indicatorset}{\@indicator}}
\def\@indicatorset[#1]{\mathbf{1}_{\{#1\}}}
\def\@indicator{\mathbf{1}}
\makeatother

\newcommand{\ysigma}{\sigma}

\newcommand{\ydim}{d}
\newcommand{\adim}{m}
\newcommand{\BMdim}{d}
\newcommand{\dpeyzspace}{\mathcal{O}}
\newcommand{\dpedom}{\dpeyzspace_T}
\newcommand{\pbdry}{\partial_T}
\newcommand{\dpeparabbdry}{\pbdry \dpedom}
\newcommand{\yspace}{\R^\ydim}
\newcommand{\zspacegen}{\mathcal{O}_z} 
\newcommand{\zspacelower}{0}
\newcommand{\zspaceupper}{c}
\newcommand{\zspace}{[\zspacelower,\zspaceupper]}

\newcommand{\zspacebdry}{\{\zspacelower,\zspaceupper\}}
\newcommand{\Hz}{H_z}
\newcommand{\contrProcSet}{\mathcal{A}}
\newcommand{\contrValSet}{\mathbb{A}}

\newcommand{\symmat}{\mathcal{S}}
\newcommand{\sndPSupDiff}{\mathcal{P}^{2,+}}
\newcommand{\sndPSubDiff}{\mathcal{P}^{2,-}}
\newcommand{\sndPSupDiffCl}{\overline{\mathcal{P}}^{2,+}}
\newcommand{\sndPSubDiffCl}{\overline{\mathcal{P}}^{2,-}}

\DeclareMathOperator{\AVaR}{AVaR}
\DeclareMathOperator{\dom}{dom}
\DeclareMathOperator{\intdom}{int\,\dom}

\begin{document}

\title{Stochastic control of optimized certainty equivalents}

\author{Julio Backhoff-Veraguas\thanks{Faculty of Mathematics, University of Vienna, Austria.}
\and
A. Max Reppen\thanks{Questrom School of Business, Boston University, Boston, MA, USA.
Partly supported by the Swiss National Science Foundation grant SNF 181815.}
\and Ludovic Tangpi\thanks{ORFE Department, Princeton University, Princeton, USA. Supported by the NSF grant DMS-2005832.}}

\date{\today}

\maketitle

\begin{abstract}
    \noindent
	Optimized certainty equivalents (OCEs) is a family of risk measures widely used by both practitioners and academics.
	This is mostly due to its tractability and the fact that it encompasses important examples, including entropic risk measures and average value at risk.

	In this work we consider stochastic optimal control problems where the objective criterion is given by an OCE risk measure, or put in other words, a risk minimization problem for controlled diffusions.
	A major difficulty arises since OCEs are often time inconsistent. Nevertheless, via an enlargement of state space we achieve a substitute of sorts for time consistency in fair generality. This allows us to derive a dynamic programming principle and thus recover central results of (risk-neutral) stochastic control theory.
	In particular, we show that the value of our risk minimization problem can be characterized as a viscosity solution of a Hamilton--Jacobi--Bellman--Issacs equation. We further establish a comparison principle and uniqueness of the latter under suitable technical conditions.
\end{abstract}

\section{Introduction and main results}

Let $T\in (0,\infty)$ be a fixed deterministic time horizon and $(\Omega, {\cal F}, P)$ a given probability space equipped with the completed filtration $({\cal F}_t)_{t\in [0,T]}$ of a $\BMdim$-dimensional Brownian motion $W$. Further let $\contrValSet \subseteq \mathbb{R}^{\adim}$ be a compact and convex set, ${\contrProcSet}$ be the set of $\contrValSet$-valued progressively measurable processes, and assume that the functions $b,\sigma$ satisfy
\begin{equation*}%
    \label{eq:ass.b}\global\def\assb{\eqref{eq:ass.b}}
    \tag{A$b\sigma$}
    \left\{
        \begin{array}{l}
        (b,\sigma):[0,T]\times \mathbb{R}^\ydim\times \mathbb{R}^\adim\to \mathbb{R}^\ydim \times \mathbb{R}^{\ydim\times\ydim}\text{ are {jointly continuous and bounded};}\\
         |b(t,y_1,a)-b(s,y_2,a)|{+\|\sigma(t,y_1,a)-\sigma(s,y_2,a)\|}\leq c_2(|t-s| + |y_1-y_2|); \\
       \text{for each $t,y$ the set }
      { K(t,y):=\left\{ (b(t,y,a),\sigma\sigma^\top(t,y,a)):\, a\in \contrValSet  \right\}
\text{ is convex,}}
    \end{array}
    \right.
\end{equation*}
where we use $\|\cdot\|$ to denote the operator norm.
In particular, under condition \eqref{eq:ass.b}, for each $\alpha \in {\contrProcSet}$, the process $Y^{y,\alpha}$ is well-defined:
\begin{equation}%
    \label{eq:control.SDE}
    dY^{y,\alpha}_t = b(t, Y^{y,\alpha}_t,\alpha_t) \,dt + \ysigma(t, Y^{y,\alpha}_t,\alpha_t)\,dW_t,\quad Y^{y,\alpha}_0 = y.
\end{equation}
For the main results in this article we will have to strengthen Assumption \eqref{eq:ass.b} by additionally assuming that 
\begin{itemize}
\item {$\sigma(t,y,a)\sigma^\top(t,y,a)> 0$ in the sense of positive definite matrices; namely that $\sigma(t,y,a)$ is \emph{non-degenerate} at each point $(t,y,a)$}.
\item $\sigma(t,y,a)=\sigma(t,y)$, i.e.\ that $\sigma$ only depends on time and space but not on the control. This will be referred to as the \emph{uncontrolled $\sigma$ case}.
\end{itemize}

Our aim is to study the optimal control of the $\ydim$-dimensional diffusion $Y$ for a cost criterion based on an optimized certainty equivalent (OCE) risk measure $\rho$. That is, for a given function $f$, we focus on the optimal control problem\footnote{As usual, a running cost can also be included by adding an extra state variable.}
\begin{equation*}%
    \label{eq:prob P}\global\def\probP{\eqref{eq:prob P}}
    \tag{P}
    \inf_{\alpha \in {\contrProcSet}}\rho(f(Y^{y,\alpha}_T)).
\end{equation*}
We assume that $f$ satisfies the condition
\begin{equation*}%
    \label{eq:ass:f}\global\def\assf{\eqref{eq:ass:f}}
    \tag{A$f$}
        f:\mathbb{R}^\ydim\to \R \quad \text{is continuous, bounded from below and with polynomial growth}. 
\end{equation*}
Problem \probP{} is a \emph{risk minimization} one, with $\rho(f(Y^{y,\alpha}_T))$ representing the riskiness of $f(Y^\alpha_T)$.
The problem is then to determine the smallest possible risk and the control $\alpha^*$ leading to it.
In order to specify $\rho$, we start with a loss function $l\colon\R\to\R$.
That is, a function satisfying the usual assumptions
\begin{equation*}%
    \label{eq:ass.l}\global\def\assl{\eqref{eq:ass.l}}
    \tag{A$l$}
    \bigg\{\begin{array}{l}
        l \text{ is increasing, convex, bounded from below with at most polynomial growth, and } \\
        l(0)=0, \,l^\ast(1)=0,\text{ and } l(x)>x \text{ for $|x|$ large enough,}
    \end{array}
\end{equation*}
where $l^*$ denotes the convex conjugate of $l$ defined as
\[ l^\ast(z):=\sup_{x\in\R} (xz-l(x)), \quad z\ge0. \]
Note that $l^*$ is valued on the extended real line.
The functional $\rho:L^0\to \R\cup \{+\infty\}$ defined by
\begin{equation}
\label{eq:oce}
	\rho(X) := \inf_{r\in \R}(\E[l(X-r )]+r)
\end{equation}
is an OCE risk measure. In this interpretation we think of $X$ as a financial/economic loss, and $\rho(X)$ represents the level of risk%
\footnote{Strictly speaking, it is $X\mapsto \rho(-X)$ that is a risk measure, but we will work with $\rho$ for notational convenience.}
associated to $X$, or the minimal capital required to make $X$ ``acceptable,'' see e.g. \cite{FS3dr} for details and \cite{ben-tal01,Roboptim} for discussions on the interpretation of OCEs. Notice that, restated for OCE risk measures, Problem \probP{} takes the form:
\begin{equation}\label{eq:Prestated}
    \inf_{\alpha \in {\contrProcSet}, \,r\in \R}(\E[l(f(Y^{y,\alpha}_T)-r )]+r).
\end{equation}

Problems of type \probP{} are sometimes called risk-sensitive decision problems to emphasize the fact that the objective is not to minimize the (linear) mathematical expectation $\E[f(Y^{y,\alpha}_T)]$, but rather a \emph{convex} risk measure (however, the literature on risk-sensitive control overwhelmingly focuses on the entropic risk measure obtained by choosing $l(x) = e^x-1$). In \emph{risk-free} optimization problems one usually defines the value function
\begin{equation*}
 \label{eq:risk free prob}\global\def\probRF{\eqref{eq:risk free prob}}
    \tag{Rf}
  \phi(t,y) = \inf_{\alpha \in {\contrProcSet}}\E[f(Y^{t,y,\alpha}_T)],
\end{equation*}
where $Y^{t, y, \alpha}$ denotes the solution of \eqref{eq:control.SDE} starting at time $t$ from $y$, and derives the associated Bellman equation. 
In contrast, in most cases \probP{} cannot be (directly) solved using Bellman's equation as is done for \probRF.
This is due to the lack of a property called \emph{time-consistency} for the operator $\rho$.
In fact, unless the loss function $l$ is linear or exponential, Bellman's principle of optimality will not apply for Problem \probP{}, hindering the use of standard stochastic control techniques to characterize the value of the problem and/or of the optimal control.

\begin{example}\label{ex:entmmv}
  For $l(x) = e^x - 1$, the OCE $\rho$ becomes the ``entropic'' risk measure.
  This is essentially the only instance satisfying \assl{} leading to a time-consistent risk measure (c.f.\ Remark \ref{rem:caveat}).
  In fact, in this case $\rho$ satisfies
   $\rho(X) = \log \E e^X$ so that, up to a logarithmic transform, Problem \probP{} reduces to a risk-free optimization problem.
   This problem is for instance considered in \cite{Bor-Mat-Sch} in the context of portfolio optimization.
   Another popular risk measure in economics (see e.g.\ \cite{MMR2006}) is the monotone mean-variance, obtained in our framework by taking $l(x) = \frac{((x +1)^+)^2 - 1}{2}$.
   This leads to a time-inconsistent problem.
\end{example}

The aim of the present work is to show that \emph{Problem \probP{} can be still tackled by stochastic control techniques, in spite of time-inconsistency.}
The core idea is to enlarge the state space of the problem and deploy the rich duality theory for risk measures. There are a number of reasons why we think this is a relevant contribution, among which:
\begin{itemize}
\item It is important to know that Problem \probP{} falls into the realm of the well-established theory of stochastic control. I.e.\ there is no need for a radically new theory to deal specifically with OCE risk minimization.
\item Our main results, Theorem \ref{thm:existence} and \ref{thm:uniqueness} below, are a consequence of this stochastic control perspective.
  Therein, we in fact identify a PDE characterizing Problem \probP{}.
  This may be the basis of a future numerical method.
\item  Our PDE will be of singular type, with a discontinuous Hamiltonian. Under suitable assumptions we are able to obtain a comparison principle and therefore the uniqueness for this PDE. This is remarkable given the singularity of the problem.
Our comparison result also provides comparison for the problem in \cite{OCE-PDE} as a special case.
\end{itemize}

We refer the reader to the subsection ``Relation with the literature'' below for a brief history on the idea of state space enlargement and for a summary on existing approaches to risk minimization.

\subsection*{Main results}

We propose that the value function of Problem \probP{}, in its incarnation \eqref{eq:Prestated}, should take the form
\begin{align}\label{eq_V_ext}
V(t,y,z) := \inf_{r\in \R,\, \alpha\in \contrProcSet}(\E[l(f(Y_T^{t,y,\alpha})-r )]+rz),
\end{align}
where the $z\in[0,\infty)$ variable stands for the extension of the state space. Note that $V(0,y,1)$ corresponds to the original problem.
Through convex duality, we will see that this is a natural guess, since it opens up a stochastic game reformulation:
\begin{equation}\label{eq:gametheo}
V(s,y,z) = \inf_{\alpha }\sup_{\beta }\E\left[f(Y^{s,y,\alpha}_T)Z^{s,z,\beta}_T - l^*(Z^{s,z,\beta}_T)\right] .
\end{equation}
This will be made precise in Proposition \ref{pro:open-markov}, but for the time being it suffices to say that $Z$ is an auxiliary controlled density process coming from the dual representation of the OCE risk measures.

Our existence result, Theorem \ref{thm:existence}, characterizes the putative value function in \eqref{eq_V_ext} as a viscosity solution of a second order PDE of Hamilton--Jacobi--Bellman--Issacs type, as can be guessed from the game-theoretic reformulation \eqref{eq:gametheo}. In some cases we prove this solution to be uniquely determined. A major difficulty we encounter is that the (Hamiltonian of the) PDE that naturally emerges from the duality theory, see Equation \eqref{eqn:dpe:int} below, 
is discontinuous.
We refer to Section \ref{sec:existence} for the definition of viscosity solutions in this setting. 

Let
\[\dpeyzspace \eqdef \mathbb{R}^\ydim \times \zspacegen \quad \text{and} \quad \dpedom \eqdef (0, T) \times \dpeyzspace\]
    and $\zspacegen := \intdom(l^*)$ be the interior of the effective domain of $l^*$.
    
    \quad

\begin{theorem}%
    \label{thm:existence}
    If assumptions \assl{}, \assf{} and \assb{} are satisfied, and $\sigma$ is uncontrolled and {non-degenerate}, then 
    it holds
    \begin{equation*}
        \inf_{\alpha\in {\contrProcSet}}\rho(f(Y^{t,y,\alpha}_T)) = V(t,y,1)
    \end{equation*}
    where $V$, defined in \eqref{eq_V_ext}, is a continuous
    viscosity solution of the Hamilton--Jacobi--Bellmann--Isaacs (HJBI) equation
    \global\def\refDPE{(DPE)}%
    \begin{empheq}[left=\hspace*{-1em}\text{\refDPE}\empheqlbrace]{alignat=4}%
        \tag{E}\label{eqn:dpe:int}
        &\begin{aligned}[b]
            -\partial_tV &- \inf_{a \in \contrValSet}b(t, y, a) \partial_yV
            - \frac{1}{2}\Tr\left(\ysigma \ysigma^\top(t,y)\partial^2_{yy}V\right)\\
            &-\sup_{\beta\in\R^d}\left (\frac{1}{2} z^2|\beta|^2\partial^2_{zz}V+ z \,\partial^2_{yz}V \ysigma(t,y) \beta\right) =0
        \end{aligned} \span \quad \text{in } {}& \dpedom, \\
        \tag{$\partial_T$E}\label{eqn:dpe.Tbdry}
        & V(T, y,z) = zf(y)-l^*(z)  & (y,z)\in {}& \dpeyzspace,\\
        \tag{$\partial_{\dpeyzspace}$E}\label{eqn:dpe:spacebdry}
        & V(t,y,z) = z\phi(t,y) - l^*(z)  & (t,y,z) \in {}& [0,T]\times \partial \dpeyzspace.
    \end{empheq}
\end{theorem}
Under slightly stronger conditions, the above value function is actually the unique viscosity solution of the dynamic programming equation, in a large class of functions.
\begin{theorem}\label{thm:uniqueness}
  If in addition to the assumptions of Theorem \ref{thm:existence} we assume that the domain of $l^*$ is compact and $f$ is linearly growing, then $V$ is the unique continuous viscosity solution with linear growth of \refDPE{}.
\end{theorem}

We will see below that in general it holds that $V(t,y,z) = \inf_{\alpha \in \mathcal{A}}\rho^{l_z}(f(Y_T^{t,y,\alpha}))$, where $\rho^{l_z}$ is the OCE with loss function $l_z(x):= l(x/z)$.
The variable $z$ comes from the density of a measure change that we use to extend the state space, thereby making the problem time-consistent.

The existence result applies to both cases in Example \ref{ex:entmmv}.
Existence and uniqueness, on the other hand, applies to the following important case:
\begin{example}\label{ex:avar}
  The average value-at-risk ($\AVaR$) is arguably one of the most used risk measures by practitioners in the financial and actuarial sectors, and by their regulators.
  It is obtained in our framework by taking 
  \begin{equation*}
    l(x) = x^+/\gamma
  \end{equation*}
  for some $\gamma \in (0,1)$.
  That is, $\rho(X)=$ $\AVaR_\gamma(X)$ is the $\AVaR$ at level $\gamma$.
  In this case, $l^*(z)=0$ if $z \in [0,1/\gamma]$ and $+\infty$ if $z \in (1/\gamma, \infty)$.
  Thus, the domain of $l^*$ is the compact interval $[0,1/\gamma]$.
  Under the standing assumptions on $(b,\ysigma)$ and as a consequence of Theorems \ref{thm:existence} and \ref{thm:uniqueness}, we have that
  \begin{equation*}
    V(s,y,z) = \inf_{\alpha \in \contrProcSet}\AVaR_{\gamma z}(zf(Y^{s,y,\alpha}_T))
  \end{equation*}
  is the unique continuous viscosity solution of the HJBI equation \refDPE{}.
  In particular, $V(0,y,1) = \inf_{\alpha \in \contrProcSet}\AVaR_{\gamma}(f(Y^{y,\alpha}_T))$. Details are given at the end of Section \ref{sec:comparison}. 
\end{example}

\begin{remark}\label{rem:caveat}
    If the the cost $f$ is bounded, then the statement of Theorem \ref{thm:existence} remains true even if the loss function $\ell$ does not satisfy the polynomial growth condition in \eqref{eq:ass.l}.
    This allows for instance to apply our result to the entropic risk measure discussed in Example \ref{ex:entmmv}.
\end{remark}
\begin{remark}
If $V$ is a classical solution of \refDPE{}, then a verification argument implies that $\alpha(t,y,z)\in \text{argmin}_{a \in \contrValSet} b(t, y, a) \partial_yV(t,y,z) $, together with $\beta(t,y,z)\in \text{argmax}_{\beta\in \mathbb R^d}\left (\frac{1}{2} z^2|\beta|^2\partial^2_{zz}V(t,y,z)+ z \,\partial^2_{yz}V(t,y,z) \ysigma(t,y) \beta\right)$, is an optimal feedback control in the extended state space. Further, Example \ref{ex:avar} is also illustrative as it highlights how solving \refDPE{} provides more information than just the optimal value of the problem: Following the discussion in  \cite[Section 2.1.2]{OCE-PDE} we have that if $V$ is differentiable, then $\partial_z \,V(0,y,1) =\inf_{\alpha \in \contrProcSet}\text{VaR}_{\gamma}(f(Y^{y,\alpha}_T))$, i.e.\ the minimization of the value-at-risk.
\end{remark}

Let us now comment on the technical difficulties that we encounter when proving {Theorems \ref{thm:existence} and} \ref{thm:uniqueness}:

\begin{remark}
An essential difficulty in our analysis is the singularity of our Hamiltonian.
Indeed, the optimization over $\beta$ causes discontinuity (and explosions) for $\partial^2_{zz} V = 0$.
This issue is overcome for existence in Theorem~\ref{thm:existence} by slightly enlarging the class of viscosity solutions with a weaker solution formulation (see e.g. \cite[Section 9]{CIL} for similar ideas).
Nevertheless, the irregularity of the Hamiltonian is still a major hurdle for uniqueness, especially in the weaker solution formulation.
In fact, the discontinuity of the PDE restricts the choice of penalization functions in the \emph{comparison} proof.
Fortunately, it is possible to construct the penalization functions in such a way that the points of interest in the proof are located where the Hamiltonian is finite. Moreover, we also make ample use of the particular structure of the PDE, in which in infimum and the supremum are separated.
\end{remark}

\subsection*{Relation with the literature}

As already mentioned, we get around the problem of time-inconsistency through an ``enlargement of the state space'' technique.
This approach probably originated (at least as far as risk-sensitive control problems as concerned) in the works \cite{PflugPichlerinconsistency1,PflugPichlerinconsistency2} on optimizations of average value-at-risk in a discrete-time model.
The present article expands on the work \cite{OCE-PDE}, where a state space enlargement was used to show that OCE risk measures can be characterized by viscosity solutions of PDEs. In  \cite{OCE-PDE}, the control is, so to speak, fixed.
By contrast, here
\emph{we further consider optimal control of OCEs and investigate uniqueness of a more general PDE than in \cite{OCE-PDE}.} Beyond \cite{OCE-PDE}, the work closest to ours is \cite{Oksendaltrivial}, proposing a related PDE solution method. However, \cite{Oksendaltrivial} starts by assuming classical solutions exist and considers a jump-diffusion framework. Arguably, our work then formalizes some of the results in \cite{Oksendaltrivial} in the case without jumps.

Other approaches to time inconsistency can be found e.g.\ in \cite{EkelandLazrak,Hu-Jin-Zhou17,Bjork-Khapka-Murgoci17} for approaches based on equilibrium strategies, in \cite{ZhouLi,Bensou-Wong-Yam,He-Hu-Obloj-Zhou} and the series of papers by \citeauthor*{Christensen20} \cite{Christensen20,Christensen19-lin,Christensen19} for approaches based on pre--committed strategies.
We further refer to \cite{Bauerleinconsistent,CTMP,Shapiro-letter} for discrete-time formulations and to \cite{MillerYang,Kar-Ma-Zha} for continuous-time formulations.

Compared to these works, the theoretical appeal of our method is that it allows us to use stochastic control arguments to deal with the risk-sensitive problem \probP{}. Moreover, Theorem \ref{thm:existence} is of practical interest since it transforms the (numerical) computation of the value of the problem \probP{} into a question of numerical approximation of a partial differential equation. For this reason, having obtained uniqueness is a crucial first step in developing a PDE-based numerical method.

The central argument allowing for the enlargement of state space and hence leading to Theorem \ref{thm:existence} is to steer the minimization problem \probP{} into a stochastic differential game through the dual representation of the risk measure~$\rho$:
\begin{equation}
\label{eq:rob rep RM}
	\rho(X) = \sup_{Z\in L^1_+}(\E[XZ] - \E[l^*(Z)]),
\end{equation}
{see e.g.\ \cite{ben-tal01}.}
Notice however that in the literature on stochastic differential games, admissible strategies are often defined on much smaller sets.
Most papers consider Elliot--Kalton strategies introduced in \cite{FlemSoug}, or ``elementary strategies'' cf.\ \cite{Sirbugames}.
Such formulations cannot be adopted here since the differential game organically emerges from the problem.
This should also shed some light on the fact that the optimization problem \probP{} is characterized by an HJBI equation, and not an Hamilton--Jacobi--Bellman (HJB) equation.
En route to the proof of our existence result, we will also show (Proposition \ref{pro:open-markov}) that
\begin{equation*}
	\inf_{\alpha\in {\contrProcSet}^{M,L}}\rho(f(Y^\alpha_T)) = V(0,y,1)
\end{equation*}
where ${\contrProcSet}^{M,L}$ is the set of \emph{Markovian controls} which are Lipschitz continuous.
In other words, the open loop control problem and the Markov control problem have the same value. 
This is a technical contribution which we also want to emphasize.
The present paper extends \cite{OCE-PDE} in which a PDE characterization of $\rho(f(Y^\alpha_T))$ was obtained, for a given (and fixed) $\alpha \in \mathcal{A}^{M,L}$.
More precisely, the paper \cite{OCE-PDE} is concerned with the evaluation of the riskiness of a given contingent claim and shows that if this claim arises from a diffusion, then its riskiness can be evaluated by solving an HJB equation.
Here we go one step further by considering the controlled case in which an agent seeks to compute the minimum risk; and we complement \cite{OCE-PDE} by deriving uniqueness of the HJBI equation characterizing the value function of the problem.

Articles dealing with optimal control of average value-at-risk include \cite{PflugPichlerinconsistency1,PflugPichlerinconsistency2,CTMP,Bauerleinconsistent} and \cite{MillerYang}.
The papers \cite{PflugPichlerinconsistency1,PflugPichlerinconsistency2,Bauerleinconsistent} present discrete-time models and propose time-consistent reformulations allowing to solve $\AVaR$ optimization problems.
The articles \cite{CTMP} and \cite{MillerYang} focus on computational issues and propose algorithms allowing to compute value functions of $\AVaR$ optimization problems despite the absence of dynamic programming principles. Interestingly, the article \cite{MillerYang} makes use of the bilevel optimization form of the (primal) problem, leveraging on a HJB PDE approach together with a gradient descent step for the outer minimization. Our article can be seen as an alternative in which we explicitly do not reduce the dimensionality of our HJBI PDE.

\subsection*{Outline} The remainder of this paper is dedicated to the proofs of our main results.
In Section \ref{sec:existence}, we prove Theorem \ref{thm:existence}.
There we also show that the open-loop and Markovian problems have the same value.
In the last section we prove a comparison theorem leading to Theorem~\ref{thm:uniqueness}.

\section{Characterization and existence}
\label{sec:existence}
This section is dedicated to the proof of Theorem~\ref{thm:existence}.
It will be split into several intermediate results.
Theorem~\ref{thm:uniqueness} is proved in the subsequent section.
For completeness, we recall the notion of viscosity solution we use in Theorem \ref{thm:existence}.
Here and in the rest of the paper, we denote by $\underline F$ and $\overline F$ the lower semicontinuous envelope and the upper semicontinuous envelope of $F$, respectively.
\begin{definition}\label{def viscosity}
Let $F:[0,T]\times\dpeyzspace\times\R\times\R\times  \R^\ydim\times\R^{(\ydim+1)\times (\ydim+1)}\to \R$ be a given function.
  An upper semicontinuous function $V:[0,T]\times  \closure{\dpeyzspace} \to \R$ is said to be a viscosity subsolution of the PDE
\begin{equation}
\label{eq:pde def}
\begin{cases}
    F(t,y,z,V, \partial_t V, \partial_y V, D^2 V) = 0&  \text{in } \dpeyzspace_T  \\
    V(T, y,z) = \psi(y,z) = zf(y)-l^*(z)  & (y,z)\in \closure{\dpeyzspace},\\
         V(t,y,z) = z\phi(t,y) - l^*(z)  & (t,y,z) \in  [0,T]\times \partial \dpeyzspace
    \end{cases}
\end{equation}
   if for all $x_0=(s_0,y_0,z_0)\in [0,T]\times\dpeyzspace$ and $\varphi \in C^2([0,T]\times\dpeyzspace)$ such that $x_0$ is a local maximizer of $V-\varphi$ and $\varphi(x_0) = V(x_0)$, if $s_0=T$ we have $V(x_0) \le \psi(y_0,z_0)$; if  $(y_0,z_0) \in\partial \dpeyzspace$ we have
  \[
      V(x_0) \le z_0\phi(s_0,y_0) - l^*(z_0)
  \]
  and otherwise
  \begin{equation*}
    \underline F(x_0,V(x_0),\partial_t\varphi(x_0), \partial_y\varphi(x_0), D^2\varphi(x_0)) \le 0.
  \end{equation*}

  A lower semicontinuous function $V$ is said to be a viscosity supersolution of \eqref{eq:pde def}
   if for all $x_0=(s_0,y_0,z_0)\in [0,T]\times\dpeyzspace$ and $\varphi \in C^2([0,T]\times\dpeyzspace)$ such that $x_0$ is a local minimizer of $v -\varphi$ and $\varphi(x_0) = V(x_0)$, if $s_0=T$ we have $v(x_0)\ge \psi(y_0,z_0)$,
   if   $(y_0,z_0) \in\partial \dpeyzspace$  we have
   \[
      V(x_0) \ge z_0\phi(s_0,y_0) - l^*(z_0)
   \]
   and otherwise
  \begin{equation*}
    \overline F(x_0,V(x_0),\partial_t\varphi(x_0), \partial_y\varphi(x_0), D^2\varphi(x_0))\ge 0.
  \end{equation*}

  A function is a viscosity solution if it is both a viscosity sub- and supersolution.
\end{definition}
It should be noted that, as shown in \cite[Lemma V.4.1]{Flem-Soner}, this definition of viscosity solutions is equivalent to the definition using sub- and superjets given in Definition \ref{eq:def jets} below.
For the equation studied here, i.e., with $F$ representing the left hand side of \eqref{eqn:dpe:int}, $F$ is already upper semicontinuous and the upper semicontinuous envelope can be omitted.
Moreover, $F$ is locally continuous around any point at which $\partial_{zz} \varphi < 0$.
Finally, for $\partial_{zz} \varphi \geq 0$, $\underline{F} \equiv - \infty$, and thus trivially satisfies the condition for subsolutions.
This is the relaxation needed for existence at the points of discontinuity $\partial_{zz} \varphi = 0$, with the burden instead shifted to the comparison proof.

Let ${\cal L}$ be the space defined by
\[
{\cal L}:= \left\{\beta:[0,T]\times \Omega\to \R^\BMdim, \text{progressively measurable and } \E\int_0^T|\beta_u|^2\,du<\infty\right\}.
\]
It is well-known (see e.g. \cite{ben-tal01}) that the functional $\rho$ admits the convex dual representation
\begin{equation*}
	\rho(X) = \sup_{Z \in L^1_+: \E[Z]=1}(\E[ZX - l^*(Z)]), \quad X \in L^\infty,
\end{equation*}
and that, by monotone convergence, the representation easily extends to random variables $X$ that are bounded from below.
Furthermore, in our Brownian filtration every random variable $Z \in L^1_+$ with $\E[Z]=1$ can be written as $Z=Z^{0,1,\beta}_T$, with
\begin{equation}
\label{eq:sde.Z}
	dZ^{s,z,\beta}_t = \beta_tZ^{s,z,\beta}_t\,dW_t, \quad Z^{s,z,\beta}_s = z, \quad \text{for some } \beta \in {\cal L}.
\end{equation}
Thus, by \eqref{eq:rob rep RM}, the value function associated to the control problem \probP{} is given by
\begin{equation}
\label{eq:value}
V(s,y,z)= \inf_{\alpha \in {\contrProcSet}}\sup_{\beta \in {\cal L}}\E{\left[ Z^{s,z,\beta}_Tf(Y^{s,y,\alpha}_T) - l^*(Z^{s,z,\alpha}_T) \right]},
\end{equation}
where $Y^{s,y,\alpha}$ denotes the solution of \eqref{eq:control.SDE} on $[s,T]$ starting from $Y_s=y$. 

The rest of the proof is devoted to showing that $V$ is a viscosity solution to \refDPE{}.
To that end, it shall be useful to restrict the optimization problem to the so-called ``Markov controls'', which we define as:
\begin{equation*}
	{\contrProcSet}^M:= \{\alpha:[0,T]\times \mathbb{R}^\ydim\to \contrValSet, \text{ Borel measurable} \},
\end{equation*}
or to the more relevant subset
\begin{equation*}
	{\contrProcSet}^{M,L}=\{\alpha:[0,T]\times \mathbb{R}^\ydim\to \contrValSet, \text{ Lipschitz continuous} \}.
\end{equation*}
The advantage of working with a control $\alpha \in {\contrProcSet}^{M,L}$ is that the associated state process $Y^\alpha$ is determined by coefficients satisfying the assumptions in \cite{OCE-PDE}. This opens the way to leverage on some of the results obtained in \cite{OCE-PDE} for the uncontrolled case.

We will also consider the subset ${\cal L}_b$ of ${\cal L}$ given by
\begin{equation*}
    {\cal L}_b:={\left\{\beta \in {\cal L}: \sup_{t\in [0,T]}|\beta_t|\in L^\infty\right\}}.
\end{equation*}

\begin{lemma}\label{lem_compactness_measurable}
    Suppose that {$(b,\ysigma)$ satisfy} \assb{}. {Take $\beta \in {\cal L}_b $ so that abusing notation, we  put $\beta=\beta(W)$ where $\beta$ is a progressive bounded process on the canonical space $C([0,T];\mathbb R^\ydim)$. 
    
    Given $I\subset [0,T],\Delta\subset \mathbb R^\ydim$ compact, we write $\Gamma$ for the set of the laws of all processes $Y^{s,y,\alpha,\beta}$ constructed on some filtered probablility space $(\tilde \Omega,\tilde{\mathcal F},\tilde  P)$ for which, for some $s\in I,y\in\Delta$, we have $Y^{s,y,\alpha,\beta}_t :=y$ if $t\leq s$ while for $t>s$:
 \begin{align*}dY^{s,y,\alpha,\beta}_t&= [b(t,Y^{s,y,\alpha,\beta}_t,\alpha_t)+\sigma\sigma^\top(t,Y^{s,y,\alpha,\beta}_t,\alpha_t)\beta_t(\tilde W)]dt+\sigma(t,Y^{s,y,\alpha,\beta}_t,\alpha_t) d\tilde W_t,
 \end{align*} 
and where $\alpha$ is some $\tilde{\mathcal F}$-progressive and $\contrValSet$-valued process while $\tilde W$ is some $(\tilde{\mathcal F},\tilde P)$-Brownian motion.
Then, for any $\kappa\geq 1$, the set $\Gamma$ is compact in the $\mathcal W_\kappa$-topology\footnote{{$P_n\to P$ in this topology if and only if $\int FdP_n\to \int FdP$ for any $F\in C(C([0,T];\mathbb R^\ydim))$ with $\sup_{\omega\in C([0,T];\mathbb R^\ydim)} \frac{|F(\omega)|}{1+\sup_{t\in[0,T]}|\omega_t|^\kappa}<\infty$. This convergence is metrized by the so-called $\kappa$-Wasserstein (hence $\mathcal W_\kappa$) distance on the space of probability measures on $C([0,T];\mathbb R^\ydim)$ for which the function $\omega\mapsto\sup_{t\in[0,T]}|\omega_t|^\kappa$ is integrable. A set $\Gamma$ is relatively compact in this topology if and only if it is tight and $\lim_{N\to\infty}\sup_{P\in\Gamma}\int_{\sup_{t\in[0,T]}|\omega_t|\geq N}\sup_{t\in[0,T]}|\omega_t|^\kappa dP(\omega)=0$. }}.}
\end{lemma}

\begin{proof}
The drift and volatility terms of $Y^{s,y,\alpha,\beta}_t$ are bounded, as well as the initial conditions $(s,y)$. It follows from e.g.\ \cite{zheng1985} that $\Gamma$ is relatively compact with respect to the weak topology induced by continuous bounded functions on $C([0,T];\mathbb R^\ydim)$. Since $(b,\sigma)$ are bounded it follows from the classical BDG inequalities that $\tilde\E[\sup_{t\leq T} |Y^{s,y,\alpha,\beta}_t|^{1+\kappa}]\leq c$ for some constant $c$ uniformly in $\Gamma$. From this it easily follows that $\Gamma$ is also relatively compact with respect to the $\mathcal W_\kappa$-topology.

To finish the proof it suffices to show that $\Gamma$ is weakly closed. Take $\{s^n,y^n\}\subset  I\times \Delta $ and $(\alpha^n)$ progressive and $\contrValSet$-valued possibly in different stochastic bases $(\tilde \Omega^n,\tilde{\mathcal F}^n,\tilde  P^n)$ respectively with a Brownian motion $\tilde W^n$,  and suppose $\text{Law}(Y^{s^n,y^n,\alpha^n,\beta})\to \mathbb Q$. By selecting a subsequence we may suppose $s^n\to s\in I$ and $y^n\to y\in \Delta$. We introduce $$\bar \Omega = C([0,T];\mathbb R^\ydim) \times C([0,T];\mathbb R^\ydim) \times  \mathcal P([0,T]\times  \contrValSet),$$ with generic elements denoted $\bar\omega=(\omega^1,\omega^2,q)$. The space $C([0,T];\mathbb R^\ydim)$ is equipped with its canonical filtration, denoted $(\mathcal F^1_t)_t$, and the space  $\mathcal P([0,T]\times \contrValSet)$ is equipped with the filtration $(\mathcal F^2_t)_t$ where $\mathcal F^2_t$ is the sigma-algebra generated by the sets $\{q(J\times G):J\subset [0,t], G\subset \contrValSet \text{ measurable}\}$, so that $\bar\Omega$ is equipped with the product filtration $\bar{\mathcal{F}}_t=\mathcal F^1_t\times\mathcal F^1_t\times \mathcal F^2_t$. We embedd $\text{Law}(Y^{s^n,y^n,\alpha^n,\beta})$ into $\mathcal P(\bar\Omega)$ by considering $Q^n:=\text{Law}(Y^{s^n,y^n,\alpha^n,\beta},\tilde W^n,A^n)$ where $A^n=1_{[0,T]}dt\delta_{\alpha^n_t}$. As the space $\mathcal P([0,T]\times \contrValSet)$  is compact, and since the second marginal of $Q^n$ is fixed, up to taking a further subsequence we may assume $Q^n\to Q$ weakly. Necessarily the first marginal of $Q$ is equal to $\mathbb Q$, the process $\omega^2$ is a $(\bar{\mathcal F},Q)$-Brownian motion, and $Q-$a.s.\ the first marginal of $q$ is Lebesgue measure on $[0,T]$.

In terms of the martingale problem, we have for all $t,a\geq 0$, all continuous bounded functions $h:\bar\Omega\to\mathbb R$ which are $\bar{\mathcal F}_t$-measurable, and all $\phi\in C(\mathbb R^\ydim)$ twice-continuously differentiable with bounded derivatives, that
\begin{align}
\label{eq_marting_pb_n}
\int h(\bar\omega)[M^{n,\phi}_{t+a}-M^{n,\phi}_{t}]dQ^n=0,
\end{align}
where $M^{n,\phi}_{t}=M^{n,\phi}_{t}(\bar\omega)$ is defined by 
\begin{align*}
  \phi(\omega^1_{s^n\vee t}) &- \int\limits_{[s^n,s^n\vee t]\times \contrValSet}\bigg\{\sum_i  b_i(r,\omega^1_t,a)\partial_i\phi(\omega^1_t)\\
  &\qquad +\sum_{i,j}(\sigma\sigma^\top(r,\omega^1_t,a))_{ij} \left[ (\beta_r(\omega^2))_j\partial_i\phi(\omega^1_t)+ \frac{1}{2}  \partial^2_{i,j}\phi(\omega^1_t) \right ] \bigg\}q(dr,da).
\end{align*}
Clearly $M^{n,\phi}_{t}$ converges uniformly to $M^\phi_t$ defined by
\begin{align*}
  \phi(\omega^1_{s\vee t}) &- \int\limits_{[s,s\vee t]\times \contrValSet}\bigg\{\sum_i  b_i(r,\omega^1_t,a)\partial_i\phi(\omega^1_t)\\
  &\qquad+\sum_{i,j} (\sigma\sigma^\top(r,\omega^1_t,a))_{ij}  \left[ (\beta_r(\omega^2))_j\partial_i\phi(\omega^1_t)+ \frac{1}{2} \partial^2_{i,j}\phi(\omega^1_t)\right ]  \bigg\}q(dr,da),
\end{align*}
the latter being jointly measurable in $\bar\omega$ and bounded. On the other hand, as a function of $(\omega^1,q)$ the same function is $Q$-almost continuous. By the last statement we mean that the measure $Q$ gives mass 1 to those $q$ for which the first marginal is Lebesgue, and these are in particular continuity points for the term  $(\omega^1,q)\mapsto \int\{\dots\}dq$ above. Finally, since the $\omega^2$-marginal is fixed, a standard Lusin argument and \eqref{eq_marting_pb_n} allow as to conclude that
\begin{align}
\label{eq_marting_pb}
\int h(\bar\omega)[M^{\phi}_{t+a}-M^{\phi}_{t}]dQ=0.
\end{align}

To finalize the proof, recall the convexity assumtion in \assb{} and the set 
$K(t,y)$ defined therein, which is not only convex but also compact by continuity of $(b,\sigma)$ and compactness of $\contrValSet$. After disintegration, we hence observe that $Q-$a.s:
$$\int_\contrValSet(b_i(r,\omega^1_r,a),\sigma\sigma^\top(r,\omega^1_r,a))q_r(da)\in K(r,\omega^1_r). $$ 
Hence a measurable selection argument as in \cite[Lemma 5]{be71} provides the existence of $\alpha$, some $\bar{\mathcal F}$-progressive and $\contrValSet$-valued process, such that 
$$\int_\contrValSet(b_i(r,\omega^1_r,a),\sigma\sigma^\top(r,\omega^1_r,a))q_r(da)= (b_i(r,\omega^1_r,\alpha_r),\sigma\sigma^\top(r,\omega^1_r,\alpha_r)).$$
Observing that $\phi(\omega^1_{s\vee t})-M^\phi_t$ can be then written as
$$ \int_s^{s\vee t}\left\{\sum_i  b_i(r,\omega^1_t,\alpha_r)\partial_i\phi(\omega^1_t)+\sum_{i,j} (\sigma\sigma^\top(r,\omega^1_t,\alpha_r))_{ij}  \left[ (\beta_r(\omega^2))_j\partial_i\phi(\omega^1_t)+ \frac{1}{2} \partial^2_{i,j}\phi(\omega^1_t)\right ]  \right\}dr,$$
we derive from \eqref{eq_marting_pb} and the correspondence between weak solutions and martingale problems, that the process $\omega^1$, which we may relabel $Y^{s,y,\alpha,\beta}$, is under $(\bar\Omega,\bar{\mathcal F},\bar P)$ just as stated in this lemma and has law $\mathbb Q$ as desired.
\end{proof}

\begin{remark}
   Note that the existence of optimal control in a weak sense for the control problem \probP{} follows as a direct consequence of the above Lemma \ref{lem_compactness_measurable}.
\end{remark}

\begin{lemma}
\label{lem:open-to-markov}
If assumptions \assb{}, \assf{}, and \assl{} are satisfied, then for every strategy $\alpha\in {\contrProcSet}$, there is a strategy $\tilde{\alpha} \in {\contrProcSet}^M$ such that $\rho(f(Y^\alpha_T)) = \rho(f(Y_T^{\tilde{\alpha}}))$, where the process $Y^{\tilde \alpha}$ is built on a possibly different probability space as in Lemma \ref{lem_compactness_measurable}.
\end{lemma}
\begin{proof}
    Given $\alpha\in \mathcal{A}$,
    We may apply \cite[Corollary 3.7]{Shrevemarginals}, which extends the original \cite[Theorem 4.6]{gyongy} to obtain the existence of {functions $\tilde b,\tilde \sigma$ such that the equation
	\begin{equation*}
		d\tilde Y_t=\tilde{b}(t,\tilde Y_t)dt+\tilde\ysigma(t,\tilde Y_t) d\tilde W_t
	\end{equation*}
	admits a weak solution (i.e.\ on some probability space with some Brownian motion $\tilde W$) with one-dimensional marginals that coincide with those of $Y^{\alpha}$.
	In particular, $\tilde{b}(t,y) := \E[b(t, Y^{\alpha}_t, \alpha_t)|Y^{\alpha}_t=y]$ and $\tilde{\sigma}\tilde{\sigma}^\top(t,y) := \E[\sigma\sigma^\top(t, Y^{\alpha}_t, \alpha_t)|Y^{\alpha}_t=y]$, so that $(\tilde b(t, y),\tilde{\sigma}\tilde{\sigma}^\top(t,y) )$ belongs to the convex compact set $K(t,y)$ introduced in Assumption \assb{}.
  Therefore, a measurable selection argument as in \cite[Lemma 5]{be71} allows to find a Borel measurable function $\tilde \alpha:[0,T]\times \mathbb{R}^\ydim\to \mathbb{A}$ such that $\tilde b(t,y) = b(t,y,\tilde \alpha(t,y))$ and $\tilde \sigma\tilde \sigma^\top(t,y) = \tilde \sigma\tilde \sigma^\top(t,y,\tilde \alpha(t,y))$.
	Thus, $\tilde \alpha \in \mathcal{A}^{M}$ and $\tilde Y = Y^{\tilde \alpha}$ with
	\begin{equation}
	\label{eq:weaksol}
		dY^{\tilde \alpha}_t = b(t,Y^{\tilde \alpha}_t, \tilde \alpha(t, Y^{\tilde\alpha}_t))\,dt +  \sigma(t,Y^{\tilde \alpha}_t, \tilde \alpha(t, Y^{\tilde\alpha}_t))\,d\tilde W_t.
	\end{equation}
	}
    Since OCE risk measures are law invariant, we conclude that $\rho(f(Y^\alpha_T)) = \rho(f(Y^{\tilde{\alpha}}_T))$.
\end{proof}

{
\begin{lemma}
\label{lem:conv.lemma}
Assume \eqref{eq:ass.b}, that $\sigma$ is uncontrolled and non-degenerate, that $\alpha$ a measurable $\contrValSet$-valued Markov control, and build $Y^\alpha$ on some probability space $(\tilde \Omega,\tilde{ \mathcal{F}},\tilde P )$. Then there is a sequence $\alpha_n$ of Lipschitz $\contrValSet$-valued Markov controls, such that
$$\tilde \E[F(Y^{y,s,\alpha_n})]\to\tilde \E[F(Y^{y,s,\alpha})], $$
for all $F:C([s,T];\R^\ydim)\to\R$ measurable and with at most polynomial growth.
\end{lemma}

\begin{proof}
Without loss of generality, take $s=0$ and $y=0$.

{
\emph{Step 0: It is enough to settle the question up to exit times:}
\\
Suppose we have found  a sequence $\alpha_n$ of Lipschitz $\contrValSet$-valued Markov controls, such that
\begin{align}
\label{eq:lim_with_R}\tilde \E[F(Y^{0,0,\alpha_n,R})]\to\tilde \E[F(Y^{0,0,\alpha,R})], 
\end{align}
for all $R>0$ and all $F:C([0,T];\R^\ydim)\to\R$ measurable and with at most polynomial growth, whereby $Y^{0,0,\alpha_n,R}$ (resp.\ $Y^{0,0,\alpha,R}$) denotes the process $Y^{0,0,\alpha_n}$ (resp.\ $Y^{0,0,\alpha}$) stopped at its first exit from the ball with center the origin and radius $R$, if this time is smaller than $T$. 

We remark that
$$
\lim_{R\nearrow \infty} \sup_{\bar\alpha\in \{\alpha_n\}_n\cup\{\alpha\}}\left |\tilde \E\left[1_{\{\|Y^{0,0,\bar \alpha}\|_{\infty}\geq R\}}\{F(Y^{0,0,\bar \alpha,R}) - F(Y^{0,0,\bar \alpha}) \} \right]\right |=0, 
$$
by the polynomial growth of $F$, the precompactness established in Lemma \ref{lem_compactness_measurable}, and the fact that, if $$\tau_R(\omega):=\inf\{t\in[0,T]:|\omega_t|\geq R\}\wedge T,$$ then for all continuous paths $\omega$ we have
$$1_{\{\sup_{t\in[0,T]}|\omega_{t\wedge \tau_R(\omega)}|\geq N\}}\sup_{t\in[0,T]}|\omega_{t\wedge \tau_R(\omega)}|^\kappa \leq 1_{\{\sup_{t\in[0,T]}|\omega_t|\geq N\}}\sup_{t\in[0,T]}|\omega_t|^\kappa .$$
From this remark, it is direct to drop the radius $R$ from  \eqref{eq:lim_with_R}.}\\

\emph{Step 1: Construction of the sequence $(\alpha_n)$:}
\\
Define $Y$ as the unique strong solution to $dY_t= \sigma(t,Y_t)d\tilde W_t$, denote $\mathbb Q:=\text{Law}(Y)$ and $ \mathbb Q_t$ the $t$-marginal of $\mathbb Q$. Call $\mu(dt,dx):=\mathbb Q_t(dx)dt$. We will build $\alpha_n$ as stipulated, such that
$\alpha_n\to\alpha$ in $L^2(\mu)$. By Lusin's theorem, for each $\varepsilon$ we find $F_\varepsilon\subset [0,T]\times\R^\ydim$ compact such that $\mu(F_\varepsilon)\geq 1-\varepsilon$ and $\alpha|_{F_\varepsilon}$ is continuous. By Tietze extension theorem, particularly in the vesion of \cite[Theorem 4.1]{dugundji1951}, we build $\alpha_\varepsilon$ as a continuous extension of $\alpha|_{F_\varepsilon}$, still $\contrValSet$-valued since $\contrValSet$ is closed and convex. Via mollification we take $\alpha_{\varepsilon,\eta}\to \alpha_\varepsilon$, locally uniformly as $\eta\to 0$, each of which is smooth and $\contrValSet$-valued, since $\contrValSet$ is convex. In particular  $\alpha_{\varepsilon,\eta}|_{F_\varepsilon}$ is Lipschitz and  $\alpha_{\varepsilon,\eta}|_{F_\epsilon}\to \alpha_\epsilon|_{F_\epsilon}=\alpha|_{F_\varepsilon}$ uniformly as $\eta\to 0$. We can now take by \cite[Theorem 1]{minty1970} $\bar{\alpha}_{\varepsilon,\eta}$ a Lipschitz extension of $\alpha_{\varepsilon,\eta}|_{F_\varepsilon}$ which is still $\contrValSet$-valued, since $\contrValSet$ is closed and convex. With the help of $\{\bar{\alpha}_{\varepsilon,\eta}\}_{\varepsilon,\eta}$ we can build a sequence $\{\alpha_n\}_n$ such that
\begin{itemize}
\item $\alpha_n$ is Lipschitz and $\contrValSet$-valued;
\item $\sup_{(t,x)\in F_{1/n}}|\alpha_n(t,x)-\alpha(t,x)| \leq 1/n$.
\end{itemize}
 From here it follows, since $\contrValSet$ is compact, that 
 $$\|\alpha_n-\alpha\|^2_{L^2(\mu)}\leq 1/n^2\times(1-1/n)+\frac{2}{n}\sup_{x\in \contrValSet}|x|^2,$$
 and so  $\alpha_n\to\alpha$ in $L^2(\mu)$. 
{We remark that by continuity, and under the non-degeneracy assumption on $\sigma$, the matrix $\sigma\sigma^\top$ is invertible and locally uniformly elliptic, the latter meaning that for all $t\in[0,T]$ and $|y|\leq R$ we have $\sigma(t,y)\sigma^\top(t,y)\geq \lambda_R I$ for some $\lambda_R>0$.} Hence the same holds for $\sigma$. 
{Thus for each $R>0$, and recalling the notation for $\tau_R$ from Step 0,}  it easily follows from these considerations, the Lipschitz property of $b$, and It\^o isometry, that
 \begin{align}
 \notag
 \lim_n &\int_0^{ \tau_R(Y)}\left\{ b(u,Y_u,\alpha_n(u,Y_u))d\tilde W_u -\frac{1}{2}| \sigma(u,Y_u)^{-1}  b(u,Y_u,\alpha_n(u,Y_u))|^2 du  \right\}\\
 \label{eq:conv.b.sig.b}
 = & \int_0^{\tau_R(Y)}\left\{ b(u,Y_u,\alpha(u,Y_u))d\tilde W_u -\frac{1}{2}| \sigma(u,Y_u)^{-1}  b(u,Y_u,\alpha(u,Y_u))|^2 du  \right\}
 \end{align}
 in $L^2$ and so, up to taking a subsequence, almost surely as well.\\

\emph{Step 2: Representation via Girsanov's transform:}
\\
{Consider the stochastic exponentials $Z^R_T:=\mathcal E(\int_0^{\tau_R(Y)} \sigma(t,Y_t)^{-1}b(t,Y_t,\alpha(t,Y_t))d\tilde W_t)$ and $Z^{n,R}_T:=\mathcal E(\int_0^{\tau_R(Y)} \sigma(t,Y_t)^{-1}b(t,Y_t,\alpha_n(t,Y_t))d\tilde W_t)$ for an arbitrary $R>0$ . By SDE estimates $\{Z_T^{n,R}\}_n$ is $L^p$-bounded for every $p\geq 2$. Selecting subsequences (thanks to Alaoglu's theorem) and a diagonalization argument gives that $Z_T^{n,R}\to Z^R$ in the weak topology of $L^p$ for all $p\ge 2$, but by \eqref{eq:conv.b.sig.b} we must have $Z^R=Z^R_T$. 
We conclude in particular that 
\begin{align}
 \label{eq:yet_another_convergence}   
 \tilde\E\left[ F(Y)Z^{n,R}_T \right ]\to \tilde\E\left[ F(Y)Z^R_T \right ],
\end{align}
}for $F$ measurable and such that $\tilde\E[|F(Y)|^q]<\infty$ for some $q\in (1,2]$. If $|F(\omega)|\leq c[1+\sup_t |\omega_t|^k]$ with $k\in\mathbb N$, then $\tilde\E[|F(Y)|^q]<\infty$ is guaranteed by the BDG inequality, as $\sigma$ is bounded. Finally, observe by Girsanov theorem that {$\tilde W_\cdot - \int_0^{\cdot\wedge \tau_R(Y)} \sigma(t,Y_t)^{-1}b(t,Y_t,\alpha(t,Y_t))dt$ (respectively $\tilde W_\cdot - \int_0^{\cdot\wedge \tau_R(Y)} \sigma(t,Y_t)^{-1}b(t,Y_t,\alpha_n(t,Y_t))dt$) is a $Z^R_T\tilde P$-Brownian motion (resp.\ a $Z_T^{n,R}\tilde P$-Brownian motion), and so {on $\{t\leq \tau_R(Y)\}$ we have that} $dY_t=b(t,Y_t,\alpha(t,Y_t))dt + \sigma(t,Y_t)dB_t $  for $B$ a $Z_T^R\tilde P$-Brownian motion (resp.\  $dY_t=b(t,Y_t,\alpha_n(t,Y_t))dt + \sigma(t,Y_t)dB^n_t $  for $B^n$ a $Z_T^{n,R}\tilde P$-Brownian motion). 
Recalling from Step 0 the notation that $R$ as a superscript means the path stopped at its exit time from the ball of said radius, it follows by uniqueness in law (see \cite[Theorem 5.6]{StVa69}) that  $\tilde\E\left[ F(Y^R)Z^n_T \right ]=\tilde\E\left[ F(Y^{0,0,\alpha_n,R}) \right ]$ and $\tilde\E\left[ F(Y^R)Z_T \right ]=\tilde\E\left[ F(Y^{0,0,\alpha,R}) \right ]$. This and \eqref{eq:yet_another_convergence}, the latter applied to $F^R(\omega):=F(\omega^R)$ to be precise, establish the validity of Step 0 and hence} conclude the proof.
\end{proof}
    }

The following proposition shows that the open-loop and the Markovian formulations of the control problem have the same value, and in fact the Markov controls can be chosen to be Lipschitz:
\begin{proposition}
\label{pro:open-markov}
    If assumptions \assb{}, \assf{} and \assl{} are satisfied, and $\sigma$ is uncontrolled and {non-degenerate}, then
	\begin{equation}
	\label{eq:borel-lips}
		V(s,y,z) = \inf_{\alpha \in {\contrProcSet}^{M,L}}\sup_{\beta \in {\cal L}_b}\E\left[f(Y^{s,y,\alpha}_T)Z^{s,z,\beta}_T - l^*(Z^{s,z,\beta}_T)\right] \quad \text{for all } (t,y,z) \in [0,T]\times \dpeyzspace.
	\end{equation}
\end{proposition}
\begin{proof}
	It was shown in \cite[Proposition 3.2]{OCE-PDE} that if $f$ is bounded and $l$ satisfies \assl{}, then
	\begin{equation*}
	\label{eq:V_sup_b}
		\rho(f(Y^{s,y,\alpha}_T)) = \sup_{\beta \in {\cal L}_b}\E\left[f(Y^{s,y,\alpha}_T)Z^{s,z,\beta}_T - l^*(Z^{s,z,\beta}_T)\right]
	\end{equation*}
	for every $s,y,\alpha$.
    Since, $f(Y^{s,y,\alpha}_T) = \lim_{n\to\infty}f(Y^{s,y,\alpha}_T)\wedge n$, it follows by monotone convergence that the above holds for $f$ bounded from below.
    Thus,
	\begin{equation}
        V(s,y,z) = \inf_{\alpha \in {\contrProcSet}}\sup_{\beta \in {\cal L}_b}\E{\left[f(Y^{s,y,\alpha}_T)Z^{s,z,\beta}_T - l^*(Z^{s,z,\beta}_T)\right]},
	\end{equation}
	from which it follows that $V$ is smaller than the r.h.s.\ in \eqref{eq:borel-lips}.

	Let us prove the reverse inequality.
	For every $\varepsilon>0$, there is $\alpha \in {\contrProcSet}$ such that
	\begin{equation*}
		V(s,y,z) \ge \sup_{\beta \in {\cal L}_b}\E\left[(f(Y^{s,y,\alpha}_T))Z^{s,z,\beta}_T - l^*(Z^{s,z,\beta}_T)\right] -\varepsilon = \rho^{l_z}(zf(Y^{s,y,\alpha}_T)) - \varepsilon,
	\end{equation*}
	where the equality follows from \cite[Proposition 2.8]{OCE-PDE} and $\rho^{l_z}$
	is the OCE corresponding to the loss function $l_z(x):= l(x/z)$.
	By Lemma \ref{lem:open-to-markov}, there is $\tilde{\alpha} \in {\contrProcSet}^M$ such that $\rho^{l_z}(zf(Y^{s,y,\alpha}_T)) = \rho^{l_z}(zf(Y^{s,y,\tilde{\alpha}}_T))$. 
    Remark that $Y^{s,y,\tilde{\alpha}}$ is constructed on some (possibly different) stochastic basis $(\tilde \Omega,\tilde{\mathcal F},\tilde P)$ and Brownian motion $\tilde W$. 
    Furthermore, by Lemma \ref{lem:conv.lemma}, there is a sequence $\alpha_n$ of $\mathbb{A}$--valued Markov Lipschitz controls such that $\tilde\E[F(Y^{s,y,\alpha_n})] \to \tilde\E[F(Y^{y,s,\alpha})]$ for every measurable, real--valued function $F$ on $C([s,T],\mathbb{R}^d)$ with at most polynomial growth.
	Using again \cite[Proposition 2.8]{OCE-PDE} and the fact that $\rho^{l_z}$ is law--invariant, there is $r \in \R$ such that denoting by $\mu_{Y^{s,y, \tilde \alpha}_T}$ the law of $Y^{s,y, \tilde \alpha}_T$, we have		
    \begin{align*}
		V(s,y,z) &\ge \rho^{l_z}(zf(Y^{s,y, \tilde \alpha}_T)) -\varepsilon \ge \int l(f(x)-r)\mu_{Y^{s,y, \tilde \alpha}_T}(dx) +zr-2\varepsilon\\
        &= \tilde\E[l(f(Y^{s,y,\tilde\alpha}_T) - r)] + zr - 2\varepsilon.
    \end{align*}
    Therefore, letting $F(\omega) = l(f(\omega_T) - r)$, which is continuous and with at most polynomial growth under the assumptions on $l$ and $f$, we have
    \begin{align*}
		V(s,y,z)&\ge \lim_{n\to \infty}\tilde\E[l(f(Y^{s,y,\alpha^n}_T)-r)] +zr - 2\varepsilon \\
            &\ge \inf_{\alpha \in {\contrProcSet}^{M,L}}\inf_{r \in \R}(\tilde\E[l(f(Y^{s,y, \alpha}_T)-r)] +zr) - 2\varepsilon\\
            &=\inf_{\alpha \in  {\contrProcSet}^{M,L}}\sup_{\beta \in {\cal L}_b} \E{\left[f(Y^{s,y,\alpha}_T)Z^{s,z,\beta}_T - l^*(Z^{s,z,\beta}_T)\right]} - 2\varepsilon.
	\end{align*}
	Dropping $\varepsilon$ yields the desired result.
\end{proof}

\begin{lemma}
\label{lem:V-continuous}
    If assumptions \assb{}, \assf{} and \assl{} are satisfied, then the function $V$ is real-valued and continuous on $[0,T]\times \dpeyzspace$.
    Moreover, it holds that
    \begin{equation}
    \label{eq:boundary con}
      V(s,y,z) = z\phi(s,y) - l^*(z)\quad \text{for all}\quad (t,y,z) \in [0,T]\times\partial\dpeyzspace .
     \end{equation}
     If the domain of $l^*$ is closed, then $V$ is continuous on $[0,T]\times \closure{\dpeyzspace}$.
\end{lemma}
\begin{proof}
  Since $f$ is bounded from below we have $V>-\infty$, and by the polynomial growth property of $f$ and $l$, and the representation \eqref{eq_V_ext}, that $V<\infty$.
  Recall that, due to the growth conditions on $b$ and $\sigma$ the random variable $Y_T^{s,y,\alpha}$ has moments of every order.

  \step[Upper semicontinuity.] \label{step:continuity:uppersemi}
  Regarding the continuity statement, let $(s^n, y^n,z^n)$ be a sequence converging to $(s,y,z)$.
  For every $\alpha \in {\cal A}$, it follows by standard stability results for SDEs see e.g. \cite[Section V.5]{protter01} that
  $Y^{s^n,y^n,\alpha}_T$ converges to $Y^{s,y,\alpha}_T$ in $L^p$ for all $p<\infty$.
  Thus, for every $r \in \mathbb{R}$, by dominated convergence, continuity of $l$ and $f$ and their polynomial growth, we have that
  \begin{align*}
    \limsup_{n\to \infty}V(s^n,y^n,z^n) &\le \limsup_{n \to \infty} \E[l(f(Y^{s^n,y^n,\alpha}_T) - r) + rz^n]\\
    &=\E[l(f(Y^{s,y,\alpha}_T) - r) + rz].
  \end{align*}
  This shows that
  \begin{equation*}
   \limsup_{n\to \infty}V(s^n,y^n,z^n)  \le V(s, y, z),
  \end{equation*}
  from which upper semicontinuity follows.

  \step[Lower semicontinuity on the interior.]
  To prove lower semicontinuity, let $(s^n, y^n,z^n)$ be a sequence converging to $(s,y,z)$.
  For every $n$, there is $\alpha^n \in {\cal A}$ such that
  \begin{equation*}
    V(s^n,y^n, z^n) \ge \E\Big[Z^{s^n,z^n,\beta}_Tf(Y^{s^n,y^n,\alpha^n}_T) -l^*(Z^{s^n,z^n,\beta}_T)\Big] -\frac 1n \quad \text{for all}\quad \beta \in {\cal L}_b.
  \end{equation*}
  Let $Q$ be the probability measure absolutely continuous with respect to $P$ and with Radon--Nikodym density $Z^{s,z,\beta}_T$ and $W^Q:= W + \int_0^\cdot\beta_u\,du$.
  By Girsanov's theorem $W^Q$ is a $Q$-Brownian motion and $Y^{s^n, y^n,\alpha^n}_t = y^n + \int_{s^n}^tb(u, Y^{s^n,y^n,\alpha^n}_u, \alpha^n_u) + \sigma(u, Y^{s^n,y^n,\alpha^n}_u, \alpha^n_u)\beta_u\,du + \int_{s^n}^t\sigma(u, Y^{s^n,y^n,\alpha^n}_u, \alpha^n_u)\,dW_u^Q$.
  Lemma \ref{lem_compactness_measurable} then ensures the existence of a control $\alpha \in {\contrProcSet}$ such that, up to a subsequence, it holds that $\E_Q[f(Y^{s^n,y^n,\alpha^n}_T)] \to \E_Q[f(Y^{s,y,\alpha}_T)]$.
  Moreover, since $Z^{s^n, y^n, \beta}_T$ converges to $Z^{s,y,\beta}_T$ in $L^2$, $\beta$ is bounded and $f$ of polynomial growth, we have $\E[Z^{s^n, y^n,\beta}_Tf(Y^{s,y,\alpha}_T)]\to \E[Z^{s,y,\beta}_Tf(Y^{s,y,\alpha}_T)]$.
  Therefore, it follows by triangular inequality that
  \begin{equation*}
    \E[Z_T^{s^n,z^n,\beta}f(Y^{s^n,y^n,\alpha^n}_T)] \to \E[Z^{s,y,\beta}_Tf(Y^{s,y,\alpha}_T)].
  \end{equation*}
  Hence, by continuity of $l^*$ on its domain, we have that
  \begin{equation*}
    \liminf_{n \to \infty}V(s^n, y^n,z^n) \ge \E[Z^{s, y, \beta}_Tf(Y_T^{s,y,\alpha}) - l^*(Z^{s,y,z}_T)],
  \end{equation*}
  and since $\beta \in {\cal L}_b$ was taken arbitrarily this allows to conclude
  \begin{equation*}
    \liminf_{n \to \infty}V(s^n, y^n,z^n) \ge V(s, y, z).
  \end{equation*}

  \step[Boundary value \eqref{eq:boundary con}.]
  Assume $s<T$ and $z \in  \partial {\zspacegen}$.
  There are $a \in [0,\infty)$ and $b \in (0,\infty]$ such that $\intdom (l^*) = (a,b)$.
  Thus, $ \partial {\zspacegen} = \{a,b\}$ if $b<\infty$ and $ \partial {\zspacegen} = \{a\}$ otherwise.
  If $a = 0$ and $b=\infty$, it is clear, by \eqref{eq:boundary con}, that $z \in  \partial {\zspacegen}$ implies, $V(s,y,z) = l^*(0)$.
  Let us assume $a>0$.
  If $\beta\in {\cal L}$ is such that $P\otimes dt(\beta_t\neq0)>0$, then $P(Z^{s,1,\beta}_T\neq 1)>0$ because otherwise, $Z^{s,1,\beta}_t = 1$ $P$-a.s. for every $t\ge s$ and thus $\beta=0$, a contradiction.
  And since $Z^{s,1,\beta}_T\neq 1$ with positive probability, it follows that $aZ^{s,1,\beta}_T\notin \dom(l^*)$ with positive probability.
  In fact, if $aZ^{s,1,\beta}_T \in \dom(l^*) \subseteq [a,\infty)$, then $aZ^{s,1,\beta}_T\ge1$.
  Since $Z^{s,1,\beta}$ is a martingale starting at 1, this implies that $Z^{s,1,\beta}_T=1$, a contradiction.
  Thus, $\E[aZ_T^{s,1,\beta} f(Y^{s,y,\alpha}_T) -l^*(aZ_T^{s,1,\beta})] = -\infty$.
  If $\beta = 0$, then $\E[aZ_T^{s,1,0} f(Y^{s,y,\alpha}_T) -l^*(aZ_T^{s,1,0})] = a\E[f(Y^{s,y, \alpha}_T)] - l^*(a)$.
  Hence,
  \begin{equation*}
    V(s,y,a) = \inf_{\alpha\in \mathcal{A}}a \E[f(Y^{s,y, \alpha}_T)] - l^*(a) = a\phi(s,y)-l^*(a).
  \end{equation*}
  The case $0<b<\infty$ is handled analogously.

  \step[Lower semicontinuity on the boundary.]
  If the domain of $l^*$ is closed,
  upper semicontinuity on $[0,T]\times \closure{\dpeyzspace}$ follows exactly as in Step \ref{step:continuity:uppersemi}.
  As to lower semicontinuity, let $(s^n,y^n,z^n)\in [0,T]\times \closure{\dpeyzspace} $ converge to $(s,y,z)$ and $z\in\partial\zspacegen$.
  Then, by definition of $\phi$:
  \begin{align*}
   	\limsup_{n\to \infty}V(s^n,y^n,z^n) \ge \limsup_{n\to \infty}z^n\phi(s^n,y^n) - l^*(z^n) = z\phi(s,y) - l^*(z) = V(s,y,z).
   \end{align*}
\end{proof}

Consider the ``approximate value function''
\begin{equation}
\label{eq:approx V}
	V^n(s,y,z):=\inf_{\alpha\in \mathcal{A}}\sup_{\beta\in {\cal L}_n} \E[Z_T^{s,z,\beta} f(Y^{s,y,\alpha}_T) -l^*(Z_T^{s,z,\beta})],
\end{equation}
with ${\cal L}_n:=\{\beta \in {\cal L}_b: |\beta|\le n \}$.

\begin{proposition}
\label{prop aprox}
    If assumptions \assb{}, \assf{} and \assl{} are satisfied, then $(V^n)$ converges pointwise to $V$.
\end{proposition}
\begin{proof}
	It is clear that $\limsup_{n\to \infty}V^n\le V$ pointwise on $[0,T]\times \yspace\times {\zspacegen}$.

	Let us prove that $\liminf_{n\to \infty}V^n\ge V$.
	Let $\beta \in {\cal L}_b$.
	There is $N$ such that $\beta \in {\cal L}_N$.
	For $n\ge N$, we can find $\alpha^n\in {\contrProcSet}$ such that putting $Y^n:=Y^{s,y,\alpha^n}$ one has
	\begin{align}
        1/n + V^n(s,y,z)\geq \E{\left[Z^{s,z,\beta}_T f(Y^n_T) -l^*(Z^{s,z,\beta}_T)\right]}. \label{eq approx}
	\end{align}
    Hence, for $\beta\in {\cal L}_n$ fixed, it follows by Lemma \ref{lem_compactness_measurable} and Girsanov's theorem that there is $\alpha \in {\contrProcSet}$ such that, up to a subsequence, $\E{\big[Z^{s,z,\beta}_T f(Y^n_T)\big]} \to \E{\big[Z^{s,z,\beta}_T f(Y^{s,y,\alpha}_T)\big]}$.
	Hence we may take limit in the $Y$'s in \eqref{eq approx} while leaving $\beta$ fixed, obtaining
    $$\liminf_n V^n(s,y,z) \geq  \E{\left[Z_T^{s,z,\beta} f(Y^{s,y,\alpha}_T) -l^*(Z_T^{s,z,\beta})\right]}. $$
	Thus, since $\beta$ was taken arbitrarily, we have
    $$ \liminf_{n} V^n(s,y,z) \geq \sup_{\beta\in {\cal L}_b} \E{\left[Z_T^{s,z,\beta} f(Y^{s,y,\alpha}_T) -l^*(Z_T^{s,z,\beta})\right]}.$$
	This yields
	\[
		\liminf_{n\to \infty}V^n(s,y,z) \ge V(s,y,z).
	\]
\end{proof}

We now have the following dynamic programming principle for the function $V$:
\begin{proposition}
\label{prop:DPP full}
    If assumption \assb{}, \assf{} and \assl{} are satisfied, and $\sigma$ is uncontrolled and non-degenerate, then
    the dynamic programming principle holds in the following form: For all $0\leq s\leq \theta\leq T$ we have
\begin{equation}
	\label{eq:fullDPP}
    V(s,y,z) = \inf_{\alpha \in {\contrProcSet}_{s,\theta}}\sup_{\beta \in {\cal L}_{b}^{s,\theta}}\E{\left[V(\theta, Y^{s,y,\alpha}_\theta, Z^{s,z,\beta}_\theta) \right]},
	\end{equation}
	where ${\contrProcSet}_{s,\theta}$ denotes the restriction of the elements in ${\contrProcSet}$ to the interval $[s,\theta]$, with a similar notation for ${\cal L}_{b}^{s,\theta}$.
	Equation \eqref{eq:fullDPP} also holds for $V^n$ (defined in \eqref{eq:approx V}) instead of $V$, with ${\cal L}_{b}^{s,\theta}$ replaced by ${\cal L}_{n}^{s,\theta}$, and defined analogously.
\end{proposition}
\begin{proof}
    By Proposition \ref{pro:open-markov} we have that
	$V(s,y,z) = \inf_{\alpha \in {\contrProcSet}^{M, L}}V^\alpha(s,y,z)$ with
	\begin{equation}
        V^\alpha(s,y,z) := 	\sup_{\beta \in {\cal L}_b}\E{\left[f(Y^{s,y,\alpha}_T)Z_T^{s,z,\beta} - l^*(Z^{s,z,\beta}_T) \right]},
	\end{equation}
	for each $\alpha \in {\contrProcSet}^{M,L}$.
	It was shown in \cite[Corollary 3.8]{OCE-PDE} that $V^\alpha$ satisfies the DPP
	\begin{equation}
		V^\alpha(s,y,z) = \sup_{\beta \in {\cal L}_b^{s,\theta}}\E[ V^\alpha(\theta, Y^{s, y, \alpha}_{\theta}, Z^{s, z, \beta}_\theta)].
	\end{equation}
	Now let $\varepsilon>0$.
	Then, there is a control $\alpha^\varepsilon \in {\contrProcSet}^{M,L}$ (depending also on $s,y,z$) such that $V(s,y,z) \ge V^{\alpha^\varepsilon}(s,y,z)-\varepsilon$.
	Thus, we have
	\begin{align*}
        V(s,y,z) &\ge \sup_{\beta \in {\cal L}_b^{s,\theta}}\E{\left[V^{\alpha^\varepsilon}(\theta, Y^{s, y, \alpha^\varepsilon}_{\theta}, Z^{s, z, \beta}_\theta) \right]} - \varepsilon\\
        &\ge \sup_{\beta \in {\cal L}_b^{s,\theta}}\E{\left[V(\theta, Y^{s, y, \alpha^\varepsilon}_{\theta}, Z^{s, z, \beta}_\theta) \right]} - \varepsilon\\ &\ge \inf_{\alpha \in {\contrProcSet}}\sup_{\beta \in {\cal L}_b^{s,\theta}}\E{\left[V(\theta, Y^{s, y, \alpha}_{\theta}, Z^{s, z, \beta}_\theta) \right]} - \varepsilon.
	\end{align*}
	Sending $\varepsilon$ to zero we conclude that the l.h.s.\ in \eqref{eq:fullDPP} is the greater one.

Let us now show the reverse inequality
  \begin{equation}
  \label{eq:halfDPP}
    V(s,y,z) \le \inf_{\alpha \in {\contrProcSet}_{s,\theta}}\sup_{\beta \in {\cal L}_b^{s,\theta}}\E{\left[V(\theta, Y^{s,y,\alpha}_\theta, Z^{s,z,\beta}_\theta) \right]}
  \end{equation}
  for all $[s,T]$-valued stopping time $\theta$.
  To that end, let $(s,y,z)\in [0,T]\times \yspace\times {\zspacegen}$, $\theta$ a $[s,T]$-valued stopping time, $\alpha \in {\contrProcSet}_{s,\theta}$ and $\beta \in {\cal L}_b$.
  Notice that the set
  \begin{equation*}
        \left\{K^\gamma:= \E{\left[Z^{\theta, Z^{s,z,\beta}_\theta, \beta}_Tf(Y_T^{\theta, Y_\theta^{s,y,\alpha}, \gamma}) - l^*(Z^{\theta, Z^{s,z,\beta}_\theta, \beta}_T)\mid {\cal F}_\theta\right]}: \gamma \in {\contrProcSet}_{\theta,T} \right\}
  \end{equation*}
  is directed downward.
  In fact, let $\gamma^1, \gamma^2 \in {\contrProcSet}_{\theta,T}$.
  Putting $\gamma_t:= \gamma_t^11_{\{K^{\gamma^1}<K^{\gamma^2}\}}+\gamma_t^21_{\{K^{\gamma^1}\ge K^{\gamma^2}\}} $ on $\{t\ge \theta\}$ and $\gamma_t = 0$ on $\{t<\theta\}$, it holds $\gamma \in {\contrProcSet}_{\theta, T}$ and $K^{\gamma} \le K^{\gamma^1}\wedge K^{\gamma^2}$.
  Thus, there is a sequence $(\gamma^n)$ in ${\contrProcSet}_{\theta, T}$ such that
    \[\lim_{n\to \infty} K^{\gamma^n} = \essinf_{\gamma \in {\contrProcSet}_{\theta,T}}K^\gamma.\]
By Lemma \ref{lem_compactness_measurable} and Girsanov's theorem, there is an admissible $\bar\gamma \in {\cal A}_{\theta,T}$ such that
\begin{equation*}
  \lim_{n\to \infty}K^{\gamma^n} = K^{\bar\gamma} \quad P\text{-a.s.}
\end{equation*}
  That is, $\essinf_{\gamma \in {\contrProcSet}_{\theta,T}}K^\gamma = K^{\bar\gamma}$.
  Using that $\bar\gamma$ is optimal, it follows that for $\bar\alpha:= \alpha1_{[0,\theta)} + \bar\gamma1_{[\theta,T]} $, one has
  \begin{align*}
        &\E{\left[Z^{s,y,\beta}_Tf(Y^{s,y,\bar\alpha}_T) - l^*(Z^{s,z,\beta}_T) \right]}= \E{\left[ \E\left[Z^{\theta, Z^{s,z,\beta}_\theta, \beta}_Tf(Y_T^{\theta, Y_\theta^{s,y,\alpha}, \bar\gamma}) - l^*(Z^{\theta, Z^{s,z,\beta}_\theta, \beta}_T) \mid {\cal F}_\theta \right] \right]}\\
        &\quad= \E{\left[\essinf_{\gamma \in {\contrProcSet}_{\theta,T}} \E{\left[Z^{\theta, Z^{s,z,\beta}_\theta, \beta}_Tf(Y_T^{\theta, Y_\theta^{s,y,\alpha}, \gamma}) - l^*(Z^{\theta, Z^{s,z,\beta}_\theta, \beta}_T) \mid {\cal F}_\theta \right]} \right]}\\
        &\quad\le \E{\left[\esssup_{\beta' \in {\cal L}^{\theta,T}_b}\essinf_{\gamma \in {\contrProcSet}_{\theta,T}} \E{\left[Z^{\theta, Z^{s,z,\beta}_\theta, \beta'}_Tf(Y_T^{\theta, Y_\theta^{s,y,\alpha}, \gamma}) - l^*(Z^{\theta, Z^{s,z,\beta}_\theta, \beta'}_T) \mid {\cal F}_\theta \right]} \right]}.
  \end{align*}
  Since $\beta$ was taken arbitrarily, the last inequality implies
  \begin{equation*}
        V(s,y,z)=\inf_{\alpha \in {\contrProcSet}}\sup_{\beta \in {\cal L}_b}\E{\left[Z^{s,y,\beta}_Tf(Y^{s,y,\alpha}_T) - l^*(Z^{s,z,\beta}_T) \right]} \le \sup_{\beta \in {\cal L}_b^{s,\theta}}\E{\left[V^n(\theta, Y^{s,y,\alpha}_\theta, Z^{s,z,\beta}_\theta) \right]}.
  \end{equation*}
  The claim then follows since $\alpha\in {\contrProcSet}_{s,\theta}$ was taken arbitrarily.

The proof for $V^n$ is the same.
 \end{proof}

\begin{lemma}
\label{lem:concave}
If \assl{} and \assf{}
 are satisfied, then for every $(t,y)$, the function $V(t,y,\cdot)$ is concave on ${\zspacegen}$.
\end{lemma}
\begin{proof}
	The proof follows from \cite[Proposition 3.3]{OCE-PDE} where it is shown that for every $\alpha\in {\contrProcSet}$ and $z>0$ it holds that
	\begin{equation*}
        \sup_{\beta \in {\cal L}_b}\E{\left[f(Y^{s,y,\alpha}_T)Z_T^{s,z,\beta} - l^*(Z^{s,z,\beta}_T) \right]} = \rho^{l_z}\left(zf(Y^{s,y,\alpha}_T)\right)
	\end{equation*}
	where $\rho^{l_z}$ is the OCE with loss function $l_z(x):=l(x/z)$.
	This representation and the definition of OCE show that
	\begin{equation*}
		V(s,y,z) = \inf_{\alpha \in {\contrProcSet}}\inf_{r \in \R}\left(\E[l(f(Y^{s,y,\alpha}_T) -r)] +zr \right)
	\end{equation*}
	from which concavity is easily derived.
\end{proof}

We can finally produce the proof of Theorem \ref{thm:existence}:

\begin{proof}[Proof of Theorem \ref{thm:existence}]
	Let us first use the inequality
	\begin{equation}
	\label{eq:halfDPP_repeated}
    V^n(s,y,z) \le \inf_{\alpha \in {\contrProcSet}_{0,\theta}}\sup_{\beta \in {\cal L}^n}\E{\left[V^n(\theta, Y^{s,y,\alpha}_\theta, Z^{s,z,\beta}_\theta) \right]}.
	\end{equation}
 to show that $V^n$ is a viscosity subsolution of the HJBI equation
	\begin{align}
	\label{eq:hjbi_n}
		\begin{cases}
        -\partial_tV^n-\inf_{a \in \contrValSet} b(t, y, a) \partial_yV^n -\frac{1}{2}\Tr\left(\ysigma\ysigma'(t, y)\partial^2_{yy}V^n\right)\\
				\qquad \qquad \qquad-\sup_{\beta\in\R^d, |\beta|\le n}\left (\frac{1}{2} z^2|\beta|^2\partial^2_{zz}V^n+ z \,\partial^2_{yz}V^n \ysigma(t, y)\beta\right) =0\\
			V^n(T,y,z) = zf(y) - l^*(z)\\
			V^n(t,y,z) = z\phi(z) - l^*(z) \text{ on } [0,T]\times \yspace \times \partial {\zspacegen}.
		\end{cases}
	\end{align}
  Hereby, we put $F^n$ the function such that the first line in the PDE \eqref{eq:hjbi_n} is given by $F^n(t, y,z, \partial_tV^n, \partial_y V^n, D^2 V^n)=0$.

	Let $\varphi \in C^2$ be a test function with bounded derivatives such that  $V^n -\varphi$ has a global maximum at $x=(s,y,z)\in [0,T]\times \yspace\times {\zspacegen}$ with $V^n(x) = \varphi(x)$.
	If $s = T$, then $\varphi(x) = zf(y) - l^*(z)$.

  If $s<T$ and $z \in \partial\zspacegen$, then it follows from Lemma \ref{lem:V-continuous} that $V(s,y,z) = z\phi(s,y) - l^*(z)$.

	Assuming $s<T$ and $z \notin  \partial {\zspacegen}$, then by \eqref{eq:halfDPP_repeated}, one has
	\begin{equation*}
        0\le \inf\limits_{\alpha \in {\contrProcSet}_{0,s+u}}\sup\limits_{\beta\in {\cal L}^n_b}\E{\left[\varphi(s+u, Y_{s+u}^{s,y}, Z^{s,z,\beta}_{s+u}) - \varphi(s,y,z) \right]}
	\end{equation*}
	for all $u\in (0,T-s)$.
	Let $\alpha \in \contrProcSet$ be arbitrary.
	Applying It\^o's formula to $t\mapsto \varphi(t, Y_{t}^{s,y,\alpha}, Z^{s,z,\beta}_{t})$ yields
	\begin{align}
  \notag
        0\,\le & \sup\limits_{\beta \in {\cal L}^n_b}\int_s^{s+u}\E\Bigl [ b(t, Y^{s,y,\alpha}_t, \alpha_t) \partial_y\varphi(t, Y^{s,y,\alpha}_t, Z^{s,z,\beta}_t) + \partial_t\varphi(t, Y_t^{s,y,\alpha}, Z^{s,z,\beta}_t)\\\notag
		& + \frac{1}{2}\Tr(\partial_{yy}\varphi(t, Y^{s,y,\alpha}_t, Z^{s,z,\beta}_t)\ysigma\ysigma'(t, Y^{s,y,\alpha}_t))+\frac{1}{2}\partial_{zz}\varphi(t, Y^{s,y,\alpha}_t, Z^{s,z,\beta}_t)|\beta_t|^2(Z_t^{s,z,\beta})^2\\
    \label{eq:ineq.for.subsol}
        & +\partial_{yz}\varphi(t, Y^{s,y,\alpha}_t, Z^{s,z,\beta}_t)\ysigma(t, Y^{s,y,\alpha}_t)\beta_tZ^{s,z,\beta}_t \Bigr ]\,dt .
	\end{align}
	Since $\varphi$ and its derivatives are Lipschitz continuous, and by Cauchy--Schwarz inequality and classical SDE estimates, there is a continuous function $t\mapsto R(t)$ with $R(0)=0$, further parametrized only by $\ysigma,s,b,\varphi,n,z,y$, such that
	\begin{align*}
        &0\, \le	 \sup\limits_{\beta \in {\cal L}^n_b}\int_s^{s+u} R(t-s)+ \E{\left[ \partial_y\varphi(t, Y^{s,y,\alpha}_t, Z^{s,z,\beta}_t) b(t, Y^{s, y, \alpha}_t,\alpha_t) + \partial_t\varphi{(t, Y^{s,y,\alpha}_t, Z^{s,z,\beta}_t)}\right]} \\
        & + \E\bigg[ \partial_{yz}\varphi{(t, Y^{s,y,\alpha}_t, Z^{s,z,\beta}_t)}\ysigma{(t, Y^{s,y,\alpha}_t)}\beta_tZ^{s,z,\beta}_t\\
        &\qquad + \frac{1}{2}\left(\Tr(\partial_{yy}\varphi(t, Y^{s,y,\alpha}_t, Z^{s,z,\beta}_t)\ysigma\ysigma'{(t, Y^{s,y,\alpha}_t)})+\partial_{zz}\varphi{(t, Y^{s,y,\alpha}_t, Z^{s,z,\beta}_t)}|\beta_t|^2(Z_t^{s,z,\beta})^2\right)\bigg ]dt.
	\end{align*}
	Observe that having a uniform bound on $\beta $ was essential here. As a consequence, we have
	\begin{align*}
        &0\le \int_s^{s+u}R(t-s) + \E\bigg[  \partial_y\varphi(t, Y^{s,y,\alpha}_t, Z^{s,z,\beta}_t)b(t, Y^{s, y, \alpha}_t,\alpha_t) + \partial_t\varphi(t, Y^{s,y,\alpha}_t, Z^{s,z,\beta}_t)\\
        &\quad +\frac{1}{2}\Tr(\partial_{yy}\varphi(t, Y^{s,y,\alpha}_t, Z^{s,z,\beta}_t)\ysigma\ysigma'(t, Y^{s,y,\alpha}_t))\bigg]\\
	 	&\quad+\E\bigg[g(t,Y^{s,y}_t)Z_t^{s,z,\beta}\\
        &\quad+\sup_{\beta \in \R^d:|\beta|\le n}\frac{1}{2}\partial_{zz}\varphi(t, Y^{s,y,\alpha}_t, Z^{s,z,\beta}_t)|\beta|^2(Z_t^{s,z,\beta})^2+\partial_{yz}\varphi(t, Y^{s,y,\alpha}_t, Z^{s,z,\beta}_t)\ysigma\beta Z_t^{s,z,\beta}\bigg]dt.
	\end{align*}
	Dividing by $u$, using dominated convergence, and letting $u$ go to $0$ gives
	\begin{equation*}
  F^n(s,y,z,\partial_t\varphi(s,y,z), \partial_y \varphi(s,y,z), D^2\varphi(s,y,z) ) \leq 0
	\end{equation*}
	showing that $V^n$ is a viscosity subsolution of \eqref{eq:hjbi_n}.

	The viscosity subsolution property of $V$ now follows by stability arguments.
	In fact, by Proposition \ref{prop aprox} and Lemma \ref{lem:V-continuous}, the sequence of continuous functions $(V^n)$ increases pointwise to the continuous function $V$.
	In combination with Dini's lemma it follows that $(V^{n})$ converges to $V$ uniformly on compacts.
  Denote by $F$ the function such that the first line in \refDPE{}  is given by $F(t, y,z, \partial_tV, \partial_yV, D^2V)=0$.

	Let us be given a test function $\varphi \in C^2$ such that $V -\varphi$ has a strict local maximum at $x_0=(s_0,y_0,z_0)\in [0,T)\times \yspace\times {\zspacegen}$.
  It can be checked using stability arguments that
the non-strict local maximum case can be obtained as a consequence of the strict case.
	Let $B_{r}(x_0):=\{x: |x-x_0|\le r\}$, with $r$ small enough so $x_0$ is the maximum of $V -\varphi$ on $B_{r}(x_0)$.
    Denote by $x_n = (s_n, y_n, z_n)$ the point at which $V^{n}-\varphi$ reaches its maximum in $B_r(x_0)$.
    We may suppose $x_n\to \bar{x}$. The uniform convergence on $B_r(x_0)$ of $V^n$ to $V$ yields $(V-\varphi)(x)= \lim (V^n-\varphi)(x) \leq \lim(V^n-\varphi)(x_n)=(V-\varphi)(\bar{x})$, and we conclude $\bar{x}=x_0$.
	As $V^{n}$ is a viscosity subsolution of \eqref{eq:hjbi_n}, $\varphi$ satisfies
	\begin{align}
	\nonumber
    \partial_t\varphi(x_n)&+\inf_{a \in \contrValSet} \bigg\{b(s_n, y_n, a) \partial_y\varphi(x_n)+\frac{1}{2}\Tr\left(\ysigma\ysigma'{(s_n,y_n)}\partial^2_{yy}\varphi(x_n)\right)\\
	\label{eq:visc_subVn}
		&+\sup_{\beta\in\R^d, |\beta|\le n}\left [\frac{1}{2} z^2|\beta|^2\partial^2_{zz}\varphi(x_n)+ z \,\partial^2_{yz}\varphi(x_n) \ysigma(s_n,y_n)\beta\right ] \bigg\}\ge0,\quad \text{for all }n\in \mathbb{N},
	\end{align}
  which implies
  \begin{equation*}
    F(x_n, \partial_t\varphi(x_n), \partial_y\varphi(x_n), D^2\varphi(x_n)) \leq 0.
  \end{equation*}
  Therefore, taking the limit inferior on both sides leads to $\underline F(x_0, \partial_t\varphi(x_0), \partial_y\varphi(x_0), D^2\varphi(x_0)) \leq 0$.

	Let us now prove the supersolution property.
  That the boundary condition is satisfied follows from Lemma \ref{lem:V-continuous}.
  It remains to check the interior condition.
	To that end, we rely on the half DPP
	\begin{equation}
        V(s, y, z) \ge \inf_{\alpha \in {\contrProcSet}}\sup_{\beta \in {\cal L}_b}\E{\left[V(\theta, Y^{s, y, \alpha}_{\theta}, Z^{s, z, \beta}_\theta) \right]}\quad \text{for all } s, y, z
	\end{equation}
  satisfied by $V$ (see Proposition \ref{prop:DPP full}).
  From this property the proof of the supersolution property follows by similar (and simpler) arguments as for the subsolution property.
  In fact, the stability argument is not needed here since after applying It\^o's formula to a test function, we obtain \eqref{eq:ineq.for.subsol} with the reverse inequality and without the supremum over $\beta$. 
\end{proof}

    Let us conclude this section by observing that it is common to write HJB equations with possibly singular Hamiltonians as our in the following form:
\begin{equation}\label{eqn:HJBminformulation}
    \widehat F(t, y,z, \partial_t V, \partial_y V, \partial^2_{yy} V, \partial^2_{zz} V, \partial^2_{yz} V) = \min \{ F, G\}(t, y,z, \partial_t V, \partial_y V, \partial^2_{yy} V, \partial^2_{zz} V, \partial^2_{yz} V) = 0
\end{equation}
for a suitable function $G$ and where $F$ is the left hand side in \eqref{eqn:dpe:int}.
The general idea behind this alternative structure appears in  \cite[Section 4.3]{Pham}.
It can be checked that choosing $G(t, y,z, \partial_t V, \partial_y V, \partial^2_{yy} V, \partial^2_{zz} V, \partial^2_{yz} V) = -\partial^2_{zz} V $, the equation \eqref{eqn:HJBminformulation} is equivalent to our formulation of viscosity solutions with upper and lower semi-continuous envelopes.

\section{Comparison}
\label{sec:comparison}
In this final section we prove the comparison principle leading to the proof of Theorem \ref{thm:uniqueness}, 
i.e., the uniqueness claim.
The following notation should simplify the exposition of Theorem~\ref{thm:comparison} below.
Let $\dpeparabbdry \eqdef (0, T] \times \yspace \times \zspacebdry \cup \{T\} \times \closure{\dpeyzspace}$ be the parabolic boundary of $\dpedom$.
For any $M \in \R^{\ydim+1 \times \ydim+1}$, we may write
\[
    M = \begin{bmatrix}
        Y & X \\
        X' & Z
    \end{bmatrix},
\]
where $Y \in \R^{\ydim \times \ydim}$, $X \in \R^{\ydim \times 1}$, and $Z \in \R$.
Hence, for $\yspace \times \R \ni (y, z) \mapsto \varphi(y, z) \in \R$ and $M = D^2 \varphi$, $Y$ is the Hessian in the $y$ variable, $Z$ is the second partial derivative in $z$, and $X$ is the vector of cross derivatives.
In the sequel, we will use the correspondence $M \leftrightarrow (Y, Z, X)$, with the understanding that it extends to diacritics and subscripts, e.g.\ $\hat{M} \leftrightarrow (\hat{Y}, \hat{Z}, \hat{X})$.

Let $F$ be the function such that \eqref{eqn:dpe:int} is given by
\[
    F(t, y,z, \partial_t V, \partial_y V, D^2 V) = 0.
\]
We note that the supremum can equivalently be taken over $\beta' = z\beta$, so we may drop the $z$-dependence from the notation.
Then, with
\[
  m_\beta(t, y) = \begin{bmatrix}
      \sigma(t,y) & \beta \\
      0 & 0
    \end{bmatrix}
  \in \R^{(\ydim + 1) \times (\ydim + 1)}
\]
and $\Hz$ defined by the second equality,
\begin{align*}
    F(t, y, p_t, p_y, M)
    &= - p_t - \inf_{a \in \contrValSet} b(t, y, a) p_y - \Hz(t, y, M), \\
    &= - p_t - \inf_{a \in \contrValSet} b(t, y, a) p_y - \sup_{\beta \in \R} \frac{1}{2} \Tr m_\beta(t,y)^\top m_\beta(t,y)  M.
\end{align*}
Recall that $\overline{\Hz}$ and $\underline{\Hz}$ denote the upper and lower semicontinuous envelopes of $\Hz$.
Lemma~\ref{lem:ishiiconsequence} establishes sufficient conditions for $\overline{\Hz}(t, y, M_u) - \Hz(s, \nu, M_v)$ to be suitably bounded.

The following lemma exploits the homogeneity of $F$ to transform \eqref{eqn:dpe:int} into a form better suited for proving comparison.

\begin{lemma}\label{lem:transformedPDE}
    If $u$ is a subsolution (supersolution) to
    $F(t, y, \partial_t u, \partial_y u, D^2 u) = 0$,
    then $e^{t} u$ is a subsolution (supersolution) to
    \begin{equation}%
        \label{eqn:properPDE}
        \tag{\ref{eqn:dpe:int}$'$}
        u + F(t, y,z, \partial_t u, \partial_y u, D^2 u) = 0.
    \end{equation}
\end{lemma}

\begin{proof}
    We prove the statement for subsolutions;
    the proof for supersolutions is analogous.
    For any $(s, \upsilon, \zeta)$, let $\varphi$ be a viscosity test function touching $e^{t} u$ from above at $(s, \upsilon, \zeta)$.
    Then $e^{-t} \varphi$ touches $u$ from above at $(s, \upsilon, \zeta)$, so, since $u$ is a viscosity subsolution,
    \begin{align*}
        0 &\geq F\bigl(t, y, \partial_t (e^{-t} \varphi)(s, \upsilon, \zeta), \partial_y e^{-s} \varphi(s, \upsilon, \zeta), D^2 e^{-s} \varphi(s, \upsilon, \zeta)\bigr) \\
        &= F\bigl(t, y, e^{-s} [-\varphi(s, \upsilon, \zeta) + \partial_t \varphi(s, \upsilon, \zeta)], e^{-s} \partial_y \varphi(s, \upsilon, \zeta), e^{-s} D^2 \varphi(s, \upsilon, \zeta)\bigr) \\
        &= e^{-s} \varphi(s, \upsilon, \zeta) + e^{-s} F\bigl(t, y, \partial_t \varphi(s, \upsilon, \zeta), \partial_y \varphi(s, \upsilon, \zeta), D^2 \varphi(s, \upsilon, \zeta)\bigr),
    \end{align*}
    where we implicitly use that $e^{-t}$ is strictly positive, so that the Hamiltonians in $F$ are not affected.
    After multiplication by $e^s$, this proves the claim.
\end{proof}

The following definition will be useful in the proof of comparison.
It mirrors the usual definition, but omits the derivatives that are not evaluated in $F$.
\begin{definition}
\label{eq:def jets}
    The so-called second order superjet, or superjet for short, of $u$ at $x = (t, y, z)$ is defined as
    \begin{align*}
        \sndPSupDiff u(x) =
        \{ (\partial_t \varphi, \partial \varphi, D^2 \varphi) :
            \varphi &\in C^2([0,T]\times\dpeyzspace) \\
            &\text{ and $u-\varphi$ has a local maximum at $x$} \}.
    \end{align*}
    As per usual, we also define
    \begin{align*}
        \sndPSupDiffCl u(x) =
        \{ &(p_t, p_y, M) \in \R \times \yspace \times \symmat(\ydim + 1)
            : \\
            &\exists (x^n, p_t^n, p_y^n, M^n) \in \dpedom \times \R \times \yspace \times \symmat(\ydim + 1) \\
            &\text{ such that }
            (p_t^n, p_y^n, M^n) \in \sndPSupDiff u(x^n) \\
            &\text{ and }
            (x^n, p_t^n, p_y^n, M^n) \to (x, p_t, p_y, M)
        \},
    \end{align*}
    where $\symmat(N)$ is the set of symmetric $N \times N$ matrices.
    Finally, define the second order subjet as $\sndPSubDiff u(x) = -\sndPSupDiff (-u)(x)$ and $\sndPSubDiffCl$ analogously.
\end{definition}
As $F$ and $\underline{F}$ are upper and lower semicontinuous, respectively, the limiting procedure in the definition of $\sndPSupDiffCl$ and $\sndPSubDiffCl$ does not pose a problem for defining viscosity solutions using the superjets and subjets.
This equivalent definition is standard, and the reader is referred to \cite{CIL} for details.

By Lemma~\ref{lem:transformedPDE}, it is clear that if \eqref{eqn:properPDE} has comparison, then so does the original equation.
In analyzing \eqref{eqn:properPDE} there remains the difficulty
that $\Hz$ is discontinuous, and in particular that it attains $\infty$.
This problem is exacerbated by the fact that $\Hz(0) = 0$, but $\overline{\Hz}(0) = \infty$.%
\footnote{By definition, $\Hz$ equals $\infty$ whenever $Z=0$, but $X \neq 0$.
Hence, by choosing any limit of $Z, X \to 0$ with these properties, it is clear that $\overline{\Hz}(0) = \infty$.}
The discontinuity problem is overcome by observing that $-\Hz$ is finite for any element in $\sndPSubDiffCl(v)$, as $v$ is a supersolution, and, at the maximizer constructed in the proof, the same holds for elements in $\sndPSupDiffCl(u)$.
The problem due to the semicontinuous envelope at $M = 0$ is overcome by slight perturbations of the penalty functions.
This has to be done with care, as otherwise the property used in handling the discontinuity of $\Hz$ fails.
These two techniques lead us to the following theorem.

\newcommand{\ydoub}{\iota}
\begin{theorem}%
    \label{thm:comparison}
    Let $u$ ($v$) be a linearly growing upper (lower) semicontinuous viscosity subsolution (supersolution) to~\eqref{eqn:properPDE} in $\dpedom$.
    If either $u$ or $v$ is continuous,
    then $u \leq v$ on $\dpeparabbdry$ implies that $u \leq v$ everywhere.
\end{theorem}

Before we begin the proof, in the following lemma we summarize one step used twice later on.

\begin{lemma}\label{lem:ishiiconsequence}
  Let $h : \yspace \to \R$ a $C^2$ function and
  \[
    A =
    \begin{bmatrix}
      D^2 h(y - \ydoub) & &  -D^2 h(y - \ydoub) & \\
      & 1/\varepsilon & & - 1/\varepsilon \\
      -D^2 h(y - \ydoub) & & D^2 h(y - \ydoub) & \\
      & -1/\varepsilon & & 1/\varepsilon
    \end{bmatrix},
  \]
  for $y, \ydoub \in \yspace$.
  Suppose $M_u$ and $M_v$ are matrices satisfying
  \[
    \begin{bmatrix}
      M_u & \\
      & -M_v
    \end{bmatrix}
    \leq
    A +
    \begin{bmatrix}
      D^2 g(y) & & & \\
      & \hphantom{-2\varepsilon} & & \\
       & & D^2 g(\ydoub) & \\
      &  & & -2\varepsilon 
    \end{bmatrix}
    + \gamma A^2,
  \]
  for some arbitrary constant $\gamma$.

  Let $\sqrt{\Lambda}$ be the bound on $\sigma$.
  That is, $\|\sigma\| \leq \sqrt{\Lambda}$.
  Whenever $\overline{\Hz}(t, y, M_u) < \infty$, it holds that, for some constant $C$ depending on $D^2 h$, $\varepsilon$, and $\Lambda$,
  \begin{enumerate}[label={(\arabic*)}]
    \item \label{ishiiconsequence_bounded}
      if $\| D^2 h \| < \infty$, then
      \[
      \overline{\Hz}(t, y, M_u) - \Hz(s, \ydoub, M_v) \leq 4 \Lambda \| D^2 h \|
      + \Lambda (\| D^2 g(y) \| + \| D^2 g(\ydoub) \|) + \gamma C.
      \]

    \item \label{ishiiconsequence_epsilon}
      if all quantities are implicitly parametrized by $\varepsilon$ such that $D^2 h(y - \ydoub) = \frac{1}{\varepsilon} I$ and $(t-s)^2 + |y - \ydoub|^2 \in o(\varepsilon)$ as $\varepsilon \to 0$, then
      \begin{equation*}
        \overline{\Hz}(t, y, M_u) - \Hz(s, \ydoub, M_v) \leq o(\varepsilon^0)
        + \Lambda (\| D^2 g(y) \| + \| D^2 g(\ydoub) \|) + \gamma C.
      \end{equation*}

  \end{enumerate}
\end{lemma}

\begin{proof}
  Using that
  \[
    m_\beta(t, y) - m_\beta(s, \ydoub) = \begin{bmatrix}
      \sigma(t,y) - \sigma(s, \ydoub) & 0 \\
      0 & 0
    \end{bmatrix},
  \]
  we multiply the matrices in the lemma by
  \[
    \begin{bmatrix}
      m_\beta(t, y)^\top m_\beta(t, y) & m_\beta(s, \ydoub)^\top m_\beta(t, y) \\
      m_\beta(t, y)^\top m_\beta(s, \ydoub) & m_\beta(s, \ydoub)^\top m_\beta(s, \ydoub)
    \end{bmatrix},
  \]
  complete the square, and
  take the trace to obtain
  \begin{multline*}
    \Tr \big[ m_\beta(t, y)^\top m_\beta(t, y) M_u - m_\beta(s, \ydoub)^\top m_\beta(s, \ydoub) M_v \big]
    \\
    \begin{aligned}
      &\leq \Tr \big[ (\sigma(t, y) - \sigma(s, \ydoub))^\top (\sigma(t, y) - \sigma(s, \ydoub)) D^2 h(y - \ydoub) - 2 \varepsilon |\beta|^2 \big] \\
        &\quad+ \Tr \big[ \sigma(t, y)^\top \sigma(t, y) D^2 g(y) + \sigma(s, \ydoub)^\top \sigma(s, \ydoub) D^2 g(\ydoub) \big] + \gamma C,
    \end{aligned}
  \end{multline*}
  where $C$ bounds the terms from $A^2$, which is possible because $A$ is bounded independently of $y$ and $\ydoub$.
  We note that $\beta$ does not appear in the terms bounded by $C$.
  Because
  \[
    \overline{\Hz}(t, y, M_u) - \Hz(s, \ydoub, M_v) \leq \sup_\beta \Tr \Big[ m_\beta(t, y)^\top m_\beta(t, y) M_u - m_\beta(s, \ydoub)^\top m_\beta(s, \ydoub) M_v \Big],
  \]
  and only the term $-2 \varepsilon |\beta|^2$ depends on $\beta$, the optimizer is $\beta = 0$.

  Part \ref{ishiiconsequence_bounded} follows directly from the assumed bounds.

  For part \ref{ishiiconsequence_epsilon},
  the $o(\varepsilon^0)$ term is obtained from the Lipschitz assumption on $\sigma$ and the assumed limiting behavior of $h$, $y$, and $\ydoub$ as $\varepsilon \to 0$.
\end{proof}

In the comparison proof that follows, we proceed in steps:
first we establish the bound \eqref{eq.compstep1}, which is then used later on to construct viscosity test functions and a contradiction.
The general structure for handling the $y$ variable follows \cite[Section 5.D]{CIL}, but must be adjusted to account for the discontinuity in the Hamiltonian.
Indeed, whereas \cite[(5.10)]{CIL} provides a bound for all variables `doubled', due to our Hamiltonian, \eqref{eq.compstep1} cannot be adjusted to include a doubling in the $z$-variable.
This causes difficulties in the subsequent steps, where the $z$-variable necessarily appears `doubled'.
The restriction that either $u$ or $v$ is continuous is used precisely for this reason, because then \eqref{eq.compstep1} holds with the local modulus of continuity added to the bound.
Because this adjustment only holds locally, it introduces a dependence between the penalization variables $\delta$ and $\varepsilon$, which is the cause for the more carfully chosen subsequence.

\begin{proof}[Proof of Theorem~\ref{thm:comparison}]
    As $\tilde{u} = u - \epsilon / t$ is also a subsolution, we have $\tilde{u} \leq v$ for the full boundary $\partial \dpedom$.
    The proof below could thus be completed for $\tilde{u}$ instead of $u$ with this stronger assumption to obtain $\tilde{u} \leq v$ in $\dpedom$, from which $u \leq v$ in $\dpedom$ follows from letting $\epsilon \to 0$.
    Hence, without loss of generality, we may assume that $u \leq v$ on $\partial \dpedom$.
    Moreover, since the domain of $l^*$ is compact, we will denote it by $[0,c]$ with $c\ge 0$.

    \step \label{comparison:linear_bound}
    We begin by showing that for some $E$ and $K$
    \begin{equation}%
        \label{eq.compstep1}
        u(t, y, z) - v(t, \ydoub, z) - 2 K |y - \ydoub|  \leq E < \infty, \quad \forall (t, y, \ydoub, z) \in [0, T] \times \yspace \times \yspace \times \zspace.
    \end{equation}
    If \eqref{eq.compstep1} holds, we are done and proceed to Step \ref{comparison:mainstep}.
    Otherwise, notice that
    the linear growth implies the existence of an $L > 0$ such that
    \[
        u(t, y, z) - v(s, \ydoub, \zeta) \leq L ( 1 + |y| + |\ydoub| ) \text{ on } [0, T] \times \dpeyzspace \times [0, T] \times \dpeyzspace.
    \]
    We use this to define the following family of functions.
    For some constant $C_\eta$ and each $R > 0$, let $\eta_R$ be a $C^2(\yspace)$ function with the properties
    \begin{enumerate*}[label=(\roman*)]
        \item $\eta_R \geq 0$,
        \item \label{eta_2L_growth} $\liminf_{|x| \to \infty} \eta_R(x) / |x| \geq 2 L$,
        \item $|D \eta_R(x)| + \| D^2 \eta_R(x) \| \leq C_\eta$,
        \item $\lim_{R \to \infty} \eta_R(x) = 0$.
    \end{enumerate*}

    Now, let
    \begin{align*}
        \Phi_K(t, y, \ydoub, z, \zeta) = u(t, y, z) - v(t, \ydoub, \zeta) &- 2 K \sqrt{1 + |y - \ydoub|^2} - \eta_R(y) - \eta_R(\ydoub) \\
        &- \frac{1}{2\varepsilon} |z - \zeta|^2 + \varepsilon \zeta^2,
    \end{align*}
    where $\varepsilon \in (0,1)$.
    By the assumption on linear growth and condition \ref{eta_2L_growth} on $\eta_R$, $\Phi_K$ attains a maximum at some point $(\hat{t}, \hat{y}, \hat{\ydoub}, \hat{z}, \hat{\zeta})$.

    Because \eqref{eq.compstep1} does not hold, at the maximum $\Phi_K(\hat{t}, \hat{y}, \hat{\ydoub}, \hat{z}, \hat{\zeta}) > E - \eta_R(\hat{y}) - \eta_R(\hat{\ydoub}) > 0$ for any $R$ large enough.
    This implies that
    \[
        \frac{1}{2\varepsilon} |\hat{z} - \hat{\zeta}|^2
        \leq u(\hat{t}, \hat{y}, \hat{z}) - v(\hat{t}, \hat{\ydoub}, \hat{\zeta}) - 2 K \sqrt{1 + |\hat{y} - \hat{\ydoub}|^2} - \eta_R(\hat{y}) - \eta_R(\hat{\ydoub}) + \varepsilon \hat{\zeta}^2,
    \]
    which is bounded in $\varepsilon$ for a fixed $R$, so $\lim_{\varepsilon \to 0} |\hat{z} - \hat{\zeta}| \to 0$.
    Hence, there exists $\bar{z}$ such that, along a subsequence in $\varepsilon \to 0$, $\hat{z}, \hat{\zeta} \to \bar{z}$.
    Furthermore, as $\Phi_K(\hat{t}, \hat{y}, \hat{\ydoub}, \hat{z}, \hat{\zeta}) \geq \max \Phi_K(t, y, \ydoub, z, z)$,
    \begin{multline*}
        \frac{1}{2\varepsilon} |\hat{z} - \hat{\zeta}|^2
        \leq u(\hat{t}, \hat{y}, \hat{z}) - v(\hat{t}, \hat{\ydoub}, \hat{\zeta}) - 2 K \sqrt{1 + |\hat{y} - \hat{\ydoub}|^2} - \eta_R(\hat{y}) - \eta_R(\hat{\ydoub}) + \varepsilon \hat{\zeta}^2 \\
        - \max_{[0, T] \times \yspace \times \yspace \times \zspace}
        \bigl(u(t, y, z) - v(t, \ydoub, z) - 2 K \sqrt{1 + |y - \ydoub|^2} - \eta_R(y) - \eta_R(\ydoub) + \varepsilon z^2\bigr),
    \end{multline*}
    which converges to 0 by upper semicontinuity, because $\hat{z}, \hat{\zeta} \to \bar{z}$.
    It follows that, by the construction of $\eta_R$, \eqref{eq.compstep1} is satisfied if and only if
    \begin{equation} \label{eq.compstep1.Phi_K.equiv}
        \lim_{R \to \infty} \lim_{\varepsilon \to 0} \Phi_K(\hat{t}, \hat{y}, \hat{\ydoub}, \hat{z}, \hat{\zeta}) < \infty.
    \end{equation}

    We now split into cases, depending on whether there exists a divergent sequence of $R$ such that
    for each fixed $R$ there exists a subsequence of $\varepsilon \to 0$ such that always either
    $(\hat{t}, \hat{y}, \hat{z})$ or $(\hat{t}, \hat{\ydoub}, \hat{\zeta})$ lies on $\partial \dpedom$.
    If so, then both sequences converge to boundary points as $\varepsilon \to 0$.
    By upper semicontinuity and that \eqref{eq.compstep1} is satisfied on the boundary, $u(\hat{t}, \hat{y}, \hat{z}) - v(\hat{t}, \hat{\ydoub}, \hat{\zeta}) - 2 K \sqrt{1 + |\hat{y} - \hat{\ydoub}|^2}$ is bounded from above for sufficiently small $\varepsilon$.
    As the bound depends only on the boundary condition, it is independent of $R$, which implies~\eqref{eq.compstep1.Phi_K.equiv} and thus also~\eqref{eq.compstep1}.

    On the other hand, if no such limit of boundary points exists, then for sufficiently large $R$ and small $\varepsilon$, both $(\hat{t}, \hat{y}, \hat{z})$ and $(\hat{t}, \hat{\ydoub}, \hat{\zeta})$ must be interior points.
    It holds for large $R$ that,
    \[
        2 K |\hat{y} - \hat{\ydoub}| \leq u(\hat{t}, \hat{y}, \hat{z}) - v(\hat{t}, \hat{\ydoub}, \hat{\zeta}).
    \]
    Since $(\hat{t}, \hat{y}, \hat{\ydoub}, \hat{z}, \hat{\zeta})$ is a maximum, by Ishii's lemma \cite[Theorem 3.2]{CIL},
    \begin{gather*}
        (\hat{p}_t, \hat{p}_y + D \eta_R (\hat{y}), \bar{Y}_u + D^2 \eta_R(\hat{y}), \bar{Z}_u, \bar{X}_u) \in \sndPSupDiffCl u(\hat{t}, \hat{y}, \hat{z}), \\
        (\hat{p}_t, \hat{p}_y - D \eta_R (\hat{\ydoub}), \bar{Y}_v - D^2 \eta_R(\hat{\ydoub}), \bar{Z}_v + 2\varepsilon, \bar{X}_v) \in \sndPSubDiffCl v(\hat{t}, \hat{\ydoub}, \hat{\zeta}),
    \end{gather*}
    for $\hat{p}_y = 2 K D_x \sqrt{1 + |x|^2}|_{x = \hat{y} - \hat{\ydoub}}$, $\bar{M}_u \leq \bar{M}_v$, and $\bar{Z}_u \leq \bar{Z}_v$.
    As $v$ is a viscosity supersolution, $F \geq -v > -\infty$.
    Consequently, this implies that $\bar{Z}_v + 2\varepsilon \leq 0$ and thus $\overline{\Hz}(\bar{M}_u)$ is finite.
    Define $\hat{M}_u$ by $(\bar{Y}_u + D^2 \eta_R(\hat{y}), \bar{Z}_u, \bar{X}_u)$ and $\hat{M}_v$ by $(\bar{Y}_v - D^2 \eta_R(\hat{y}), \bar{Z}_v + 2\varepsilon, \bar{X}_v)$.
    Then $\hat{M}_u$ and $\hat{M}_v$ satisfy the assumptions of Lemma~\ref{lem:ishiiconsequence} for some $\gamma$.
    In particular, for each $\gamma$, \cite[Theorem 3.2]{CIL} gives a pair $\hat{M}_u$ and $\hat{M}_v$ with these properties, so $\gamma C$ may be chosen as $o(\varepsilon^0)$.

    Hence, by the viscosity properties of $u$ and $v$ as well as Lemma~\ref{lem:ishiiconsequence}\ref{ishiiconsequence_bounded},
    \begin{align*}
        u(\hat{t}, \hat{y}, \hat{z}) - v(\hat{t}, \hat{\ydoub}, \hat{\zeta})
        &\leq F\bigl(\hat{t}, \hat{\ydoub}, \hat{p}_t, \hat{p}_y - D \eta_R (\hat{\ydoub}), \hat{M}_v\bigr)
        - \underline{F}\bigl(\hat{t}, \hat{y}, \hat{p}_t, \hat{p}_y + D \eta_R (\hat{y}), \hat{M}_u\bigr) \\
        &\leq \sup_{a \in \contrValSet} b(\hat{t}, \hat{y}, a) (\hat{p}_y + D \eta_R(\hat{y})) - \sup_{a \in \contrValSet} b(\hat{t}, \hat{\ydoub}, a) (\hat{p}_y - D \eta_R(\hat{\ydoub})) \\
        &\qquad +  4 \Lambda \| D^2 h \|
          + \Lambda ( \| D^2 \eta_R(\hat{y}) \| + \| D^2 \eta_R(\hat{\ydoub}) \|)
          + o(\varepsilon^0)
        \leq E'.
    \end{align*}
    As $\hat{p}_y$ is bounded independently of $R$ and $\varepsilon$, $E'$ depends on the model parameters, $c_1$, $c_2$, $\Lambda$, and $C_\eta$, and is thus independent of $R$ and $\varepsilon$.
    First letting $\varepsilon \to 0$ and then $R \to \infty$, it follows that \eqref{eq.compstep1.Phi_K.equiv} holds.
    This proves~\eqref{eq.compstep1}.

    \step \label{comparison:mainstep}
    Suppose that there is a point $(\bar{t}, \bar{y}, \bar{z})$ such that $u(\bar{t}, \bar{y}, \bar{z}) - v(\bar{t}, \bar{y}, \bar{z}) = 2 \lambda > 0$.
    Let
    \begin{align*}
        \Phi(t, s, y, \ydoub, z, \zeta)
        = u(t, y, z) - v(s, \ydoub, \zeta)
        &- \frac{1}{2 \varepsilon} \bigl(|t - s|^2 + |y - \ydoub|^2 + |z - \zeta|^2\bigr) \\
        &- \delta (y^2 + \ydoub^2)
        + \varepsilon \zeta^2,
    \end{align*}
    for parameters $\delta \in (0,1)$ and $\varepsilon \in (0,1)$.
    Then, for sufficiently small $\delta > 0$ and $\varepsilon > 0$, $\Phi(\bar{t}, \bar{t}, \bar{y}, \bar{y}, \bar{z}, \bar{z}) \geq \lambda$.
	Because of the linear growth and the quadratic penalty, the semicontinuous function $\Phi$ attains a positive maximum at a point $(\hat{t}, \hat{s}, \hat{y}, \hat{\ydoub}, \hat{z}, \hat{\zeta})$,
    and the maximizers are bounded for each $\delta$, uniformly in $\varepsilon$.

    We will now use the continuity assumption.
    It is clear from the arguments that follow, that it does not matter whether $u$ or $v$ is continuous, so without loss of generality, let $u$ be continuous.
    Then, by continuity and that the maximizers lie in a compact domain $\dpedom^\delta \times \dpedom^\delta$, there exist moduli of continuity $m_\delta$ such that
    \[
        |u(\hat{t}, \hat{y}, \hat{z}) - u(\hat{s}, \hat{y}, \hat{\zeta})| \leq m_\delta(|\hat{t} - \hat{s}|, |\hat{z} - \hat{\zeta}|),
    \]
    for each $\delta > 0$ and all $\varepsilon > 0$.
    As a consequence of this and \eqref{eq.compstep1},
    \begin{align*}
        \frac{1}{2 \varepsilon} \bigl(|\hat{t} - \hat{s}|^2 + |\hat{y} - \hat{\ydoub}|^2 + |\hat{z} - \hat{\zeta}|^2\bigr) &+ \delta (\hat{y}^2 + \hat{\ydoub}^2)
        \leq u(\hat{t}, \hat{y}, \hat{z}) - v(\hat{s}, \hat{\ydoub}, \hat{\zeta}) + \varepsilon \hat{\zeta}^2 \\
        &\leq E + 2 K |\hat{y} - \hat{\ydoub}|
        + m_\delta(|\hat{t} - \hat{s}|, |\hat{z} - \hat{\zeta}|)
        + \varepsilon \hat{\zeta}^2 \\
        &\leq E + \frac{1}{4\varepsilon} |\hat{y} - \hat{\ydoub}|^2 + 4 \varepsilon K^2
        + m_\delta(|\hat{t} - \hat{s}|, |\hat{z} - \hat{\zeta}|)
        + \varepsilon \hat{\zeta}^2.
    \end{align*}
    It is thus clear that, for any $\delta$, $\lim_{\varepsilon \to 0} |\hat{t} - \hat{s}| + |\hat{y} - \hat{\ydoub}| + |\hat{z} - \hat{\zeta}| = 0$.
    As a consequence, for each $n \in \N$ there exists $\delta_n$ and $\varepsilon_n$ such that the right hand side is bounded and hence can be chosen so that $\delta_n (\hat{y} + \hat{\ydoub}) \leq 1/n$ for $\varepsilon \leq \varepsilon_n$.
    Finally, for each $\delta$, the right hand side of
    \begin{multline*}
        \frac{1}{2 \varepsilon} \bigl(|\hat{t} - \hat{s}|^2 + |\hat{y} - \hat{\ydoub}|^2 + |\hat{z} - \hat{\zeta}|^2\bigr) \leq
        u(\hat{t}, \hat{y}, \hat{z}) - v(\hat{t}, \hat{\ydoub}, \hat{\zeta}) - \delta (\hat{y}^2 + \hat{\ydoub}^2) + \varepsilon \hat{\zeta}^2 \\
        - \max_{\dpedom^\delta} (u(t, y, z) - v(t, y, z) - \delta(y^2 + \ydoub^2) + \varepsilon z^2)
    \end{multline*}
    is vanishing along a subsequence of $\varepsilon \to 0$, so we may pick $\varepsilon_n$ such that
    $\frac{1}{2 \varepsilon_n} (|\hat{t} - \hat{s}|^2 + |\hat{y} - \hat{\ydoub}|^2) \leq 1/n$.

    \step
    We now split into two cases depending on whether there exists a $\delta$ for which there is a sequence $(\varepsilon_n)_{n \in \N}$ converging to $0$ such that either $(\hat{t}, \hat{y}, \hat{z})$ or $(\hat{s}, \hat{\ydoub}, \hat{\zeta})$ lies on $\partial \dpedom$ for each $n$.
    Notice that as $\delta$ is fixed, they lie in a bounded subset of $\partial \dpedom$.
    Thus, along a subsequence, $(\hat{t}, \hat{s}, \hat{y}, \hat{\ydoub}, \hat{z}, \hat{\zeta})$ converges to $(\tilde{t}, \tilde{t}, \tilde{y}, \tilde{y}, \tilde{z}, \tilde{z})$ as $\varepsilon \to 0$.
    By the boundary conditions,
    \[
        0 < \lambda \leq \Phi(\tilde{t}, \tilde{t}, \tilde{y}, \tilde{y}, \tilde{z}, \tilde{z}) \leq u(\tilde{t}, \tilde{y}, \tilde{z}) - v(\tilde{t}, \tilde{y}, \tilde{z}) \leq 0,
    \]
    which is a contradiction.

    \step
    Otherwise, there exists a sequence $(\varepsilon_n, \delta_n)_{n \in \N}$, converging to $(0, 0)$, for which
    both $(\hat{t}, \hat{y}, \hat{z})$ and $(\hat{s}, \hat{\ydoub}, \hat{\zeta})$
    remain in the interior
    and, by the observations at the end of \stepref{comparison:mainstep},
    $\delta_n (\hat{y} + \hat{\ydoub}) < 1/n$, and
    $\frac{1}{2 \varepsilon_n} (|\hat{t} - \hat{s}|^2 + |\hat{y} - \hat{\ydoub}|^2) \leq 1/n$.
    At each maximizer, by Ishii's lemma \cite[Theorem 3.2]{CIL},
    \begin{gather*}
        (\hat{p}_t, \hat{p}_y + 2 \delta_n \hat{y}, \bar{Y}_u + 2 \delta_n I, \bar{Z}_u, \bar{X}_u) \in \sndPSupDiffCl u(\hat{t}, \hat{y}, \hat{z}), \\
        (\hat{p}_t, \hat{p}_y - 2 \delta_n \hat{\ydoub}, \bar{Y}_v - 2 \delta_n I, \bar{Z}_v + 2\varepsilon, \bar{X}_v) \in \sndPSubDiffCl v(\hat{s}, \hat{\ydoub}, \hat{\zeta}),
    \end{gather*}
    with $\hat{p}_y = \frac{1}{\varepsilon_n} |\hat{y} - \hat{\ydoub}|$, $\bar{M}_v \geq \bar{M}_u$, and $\bar{Z}_u \leq \bar{Z}_v$.
    Like in \stepref{comparison:linear_bound}, $v$ being a supersolution implies $\bar{Z}_v \leq -2\varepsilon$, which ensures that $\overline{\Hz}(\bar{M}_u) < \infty$.
    Again, define $\hat{M}_u$ by $(\bar{Y}_u + 2 \delta_n I, \bar{Z}_u, \bar{X}_u)$ and $\hat{M}_v$ by $(\bar{Y}_v - 2 \delta_n I, \bar{Z}_v + 2\varepsilon, \bar{X}_v)$.
    Repeating the same arguments as previously,
    $\hat{M}_u$ and $\hat{M}_v$ satisfy the assumptions of Lemma~\ref{lem:ishiiconsequence} with $\gamma C \in o(\varepsilon^0)$.

    Using the viscosity properties of $u$ and $v$ again and Lemma~\ref{lem:ishiiconsequence}\ref{ishiiconsequence_epsilon},
    \begin{align*}
        \lambda
        &\leq u(\hat{t}, \hat{y}, \hat{z}) - v(\hat{s}, \hat{\ydoub}, \hat{\zeta}) \\
        &\leq
        F(\hat{s}, \hat{\ydoub}, \hat{p}_t, \hat{p}_y - 2\delta_n \hat{\ydoub}, \hat{M}_v)
        - \underline{F}(\hat{t}, \hat{y}, \hat{p}_t, \hat{p}_y + 2\delta_n \hat{y}, \hat{M}_u) \\
        &\leq - \inf_{a \in \contrValSet} b(\hat{s}, \hat{\ydoub}, a) (\hat{p}_y - 2 \delta_n \hat{\ydoub})
        + \inf_{a \in \contrValSet} b(\hat{t}, \hat{y}, a) (\hat{p}_y + 2 \delta_n \hat{y})
        + 2 \delta_n \Lambda + o(\varepsilon^0) \\
        &\leq c_2 (|\hat{t} - \hat{s}| + |\hat{y} - \hat{\ydoub}|) (\hat{p}_y + 2 \delta_n \hat{y}) + \sup_{a \in \contrValSet} 4 c_1 (1 + |a|) \delta_n (\hat{\ydoub} + \hat{y}) + 2 \delta_n \Lambda + o(\varepsilon^0).
    \end{align*}
    Since the right hand side vanishes as $n \to \infty$,
    it follows that
    \[
        \lambda \leq \lim_{n \to \infty} u(\hat{t}, \hat{y}, \hat{z}) - v(\hat{s}, \hat{\ydoub}, \hat{\zeta}) \leq 0,
    \]
    which is a contradiction.
\end{proof}

The conclusion that $V$ is the unique viscosity solution in this class is obtained by twice comparing $V$ with any other candidate solution $W$ using Theorem~\ref{thm:comparison} to conclude that $V \leq W \leq V$.
This procedure yields the following corollary.

\begin{corollary}
    \label{cor:uniqueness}
    The value function $V$ is the unique viscosity solution of linear growth.
\end{corollary}

\bibliographystyle{abbrvnat}


\end{document}